\documentclass[11pt]{amsart}

\usepackage{palatino, mathpazo}
\usepackage[mathscr]{eucal}
\usepackage{amssymb}
\usepackage[all]{xy}
\usepackage{graphicx}
\usepackage{texdraw}
\usepackage{color}

\renewcommand{\theequation}{\arabic{section}.\arabic{equation}}
\newtheorem{theorem}{Theorem}[section]
\newtheorem{lemma}[theorem]{Lemma}

\newtheorem{corollary}[theorem]{Corollary}

\newtheorem{definition}[theorem]{Definition}
\theoremstyle{remark}
\newtheorem{remark}[theorem]{Remark}
\newtheorem{example}[theorem]{Example}
\numberwithin{equation}{section}

\newcommand{\Om}{\Omega}
\newcommand{\om}{\omega}
\newcommand{\iy}{\infty}
\newcommand{\dd}{\delta}

\newcommand{\ga}{\gamma}
\newcommand{\la}{\lambda}

\newcommand{\R}{\mathbb R}
\newcommand{\al}{\alpha}

\newcommand{\tu}{\tilde u}

\newcommand{\bu}{\bar u}

\newcommand{\bb}{\beta}

\newcommand{\e}{\varepsilon}

\newcommand{\bt}{\begin{theorem}}
\newcommand{\et}{\end{theorem}}
\newcommand{\bl}{\begin{lemma}}
\newcommand{\el}{\end{lemma}}
\newcommand{\bd}{\begin{definition}}
\newcommand{\ed}{\end{definition}}
\newcommand{\bc}{\begin{corollary}}
\newcommand{\ec}{\end{corollary}}
\newcommand{\bx}{\begin{example}}
\newcommand{\ex}{\end{example}}
\newcommand{\br}{\begin{remark}}
\newcommand{\er}{\end{remark}}
\newcommand{\be}{\begin{equation}}
\newcommand{\ee}{\end{equation}}

\begin{document}

\title[bubbling solutions to a competitive Chern-Simons model]
{A new type of non-topological bubbling solutions
to a competitive Chern-Simons model}

\author{Zhijie Chen}
\address{Department of Mathematical Sciences, Yau Mathematical Sciences Center, Tsinghua University, Beijing 100084, China}
\email{zjchen@math.tsinghua.edu.cn}
\author{Chang-Shou Lin}
\address{Taida Institute for Mathematical Sciences (TIMS), Center for Advanced Studies in Theoretic Sciences (CASTS), National Taiwan University, Taipei}
\email{cslin@math.ntu.edu.tw}

\thanks{Mathematics Subject Classification: 81T13, 35J47.\\
Emails: zjchen@math.tsinghua.edu.cn (Chen); cslin@math.ntu.edu.tw (Lin)}

\maketitle

\begin{abstract}
We study a non-Abelian Chern-Simons system in $\mathbb{R}^2$, including the simple Lie algebras $A_2$ and $B_2$. In a previous work, we proved the existence of radial non-topological solutions with prescribed asymptotic behaviors via the degree theory. We also constructed a sequence of bubbling solutions with only one component blowing up partially at infinity. In this paper, we construct a sequence of radial non-topological bubbling solutions of another type via the shooting argument. One component of these bubbling solutions locally converge to a non-topological solution of the Chern-Simons-Higgs scalar equation, but both components blow up partially in different regions at infinity at the same time.
This generalizes a recent work by Choe, Kim and the second author, where the $SU(3)$ case (i.e. $A_2$) was studied. Our result is new even for the $SU(3)$ case and also confirms the difference between the $SU(3)$ case and the $B_2$ case.
\end{abstract}

\section{Introduction}
\renewcommand{\theequation}{1.\arabic{equation}}

In this paper, we study a non-Abelian Chern-Simons system of rank 2:
\be\label{eq1-1}
\begin{pmatrix}
    \Delta u_1\\
    \Delta u_2
  \end{pmatrix}
+K\begin{pmatrix}
              e^{u_1} \\
              e^{u_2} \\
            \end{pmatrix}
          -K\begin{pmatrix}
                       e^{u_1} & 0 \\
                       0 & e^{u_2}
                     \end{pmatrix}
          K\begin{pmatrix}
              e^{u_1} \\
              e^{u_2} \\
            \end{pmatrix}
          =\begin{pmatrix}
              4\pi N_1\dd_0 \\
              4\pi N_2\dd_0
            \end{pmatrix}\;\,\text{in $\R^2$},
\ee
where $N_1, N_2$ are non-negative integers, $\dd_0$ denotes the Dirac measure at $0$, and $K=(a_{ij})$ is a $2\times 2$ matrix satisfying
\be\label{eq1-2}a_{11}, a_{22}>0,\;\, a_{12}, a_{21}<0\;\,\text{and}\;\,  a_{11}a_{22}-a_{12}a_{21}>0.\ee
Clearly \eqref{eq1-1} can be considered as a perturbation of the following Liouville system with a singular source:
\be\label{toda}
\begin{pmatrix}
    \Delta u_1\\
    \Delta u_2
  \end{pmatrix}
+\begin{pmatrix}
                      a_{11} & a_{12} \\
                       a_{21} & a_{22}
                     \end{pmatrix}\begin{pmatrix}
              e^{u_1} \\
              e^{u_2} \\
            \end{pmatrix}=\begin{pmatrix}
              4\pi N_1\dd_0 \\
              4\pi N_2\dd_0
            \end{pmatrix}\;\;\text{in $\R^2$}.
\ee
See \cite{ALW1, ALW2, CI}. Under the assumption \eqref{eq1-2}, then in literature, \eqref{toda} is also said to be {\it competitive}, as compared to the {\it cooperative} case where $a_{12}, a_{21}>0$.

In the last few decades, various Chern-Simons field theories \cite{D0} have been widely studied, largely motivated by their applications to the physics of high critical
temperature superconductivity. Another interesting feature of Chern-Simons field theories is that they provide a gauge invariant mechanism of mass
generation \cite{DJT}. These Chern-Simons theories can be reduced to systems of nonlinear partial differential equations, which have posed many mathematically
challenging problems to analysts.
Our first motivation to analyse \eqref{eq1-1} comes from the relativistic non-Abelian self-dual Chern-Simons model, which was proposed by Kao and Lee \cite{KL} and Dunne \cite{D1, D2}. Following \cite{D1, D2}, the relativistic non-Abelian self-dual Chern-Simons model is defined in the $(2+1)$ Minkowski space $\R^{1,2}$, and the gauge group is a compact Lie group with a semi-simple Lie algebra $\mathcal{G}$. The Chern-Simons Lagrangian action density $\mathcal{L}$ in $2+1$ dimensional spacetime involves the Higgs field $\phi$ and the $\mathcal{G}$-valued gauge field $A=(A_0, A_1, A_2)$. In general, the Euler-Lagrangian equation of $\mathcal{L}$ is too difficult to deal with. Therefore, we only consider the energy minimizers of the Lagrangian functional, which turn out to be the solutions of the following self-dual Chern-Simons equations:
\begin{equation}\label{eq1-3}
D_-\phi=0,\quad
F_{+-}=\frac1{\kappa^2}[\phi-[[\phi, \phi^\dag], \phi], \phi^\dag],
\end{equation}
where $D_-=D_1-iD_2$, $\kappa>0$, $F_{+-}=\partial_+A_--\partial_-A_++[A_+, A_-]$ with $A_\pm=A_1\pm i A_2$, $\partial_{\pm}=\partial_{1}\pm i\partial_2$ and $[\cdot, \cdot]$ is the Lie bracket over $\mathcal{G}$.
In \cite{D2}, Dunne considered a simplified form of the self-dual system (\ref{eq1-3}), in which the fields $\phi$ and $A$
are algebraically restricted:
$$\phi=\sum_{a=1}^r\phi^aE_a,$$
where $r$ is the rank of the gauge Lie algebra, $E_a$ is the simple root step operator, and $\phi^a$ are complex-valued functions. Let
$$u_a=\log|\phi^a|,\quad a=1,\cdots, r.$$
Then system (\ref{eq1-3}) can be reduced to the following system of nonlinear partial differential equations
\begin{align}\label{eq1-4}
\Delta u_a+\frac{1}{\kappa^2}\left(\sum_{b=1}^rK_{ab} e^{u_b}
-\sum_{b=1}^r\sum_{c=1}^r e^{u_b}K_{bc}e^{u_c}K_{ac}\right)=4\pi\sum_{j=1}^{N_a}\dd_{p_j^a},\,1\leq a\leq r,
\end{align}
where $K=(K_{ab})$ is the Cartan matrix of a semi-simple Lie algebra, $\{p_j^a\}_{j=1}^{N_a}$ are zeros
of $\phi^a$ $(a=1,\cdots, r)$, and $\dd_p$ denotes the Dirac measure concentrated at $p$ in $\R^2$. See \cite{Y1} for the derivation of (\ref{eq1-4}) from (\ref{eq1-3}). For example, there are three types of Cartan
matrices of rank $2$, which correspond to the semi-simple Lie algebras $A_2$, $B_2$ and $G_2$ respectively:
\begin{equation}\label{eq1-5}\mathcal{A}_2(\text{i.e. SU(3)})=\begin{pmatrix}2&-1\\-1&2\end{pmatrix},\;
\mathcal{B}_2=\begin{pmatrix}2&-1\\-2&2\end{pmatrix},
\;\mathcal{G}_2=\begin{pmatrix}2&-1\\-3&2\end{pmatrix}.\end{equation}

Let $((K^{-1})_{ab})$ denote the inverse of the matrix $K$. Assume that
\be\label{eq1-6}\sum_{b=1}^{r}(K^{-1})_{ab}>0, \quad a=1, 2, \cdots, r.\ee
A solution $\mathbf{u}=(u_1, \cdots, u_r)$ of \eqref{eq1-4} is called a {\it topological solution} if
$$u_a(x)\to \ln\sum_{b=1}^r(K^{-1})_{ab}\quad\text{as $|x|\to +\iy$},\quad a=1, \cdots, r;$$
a solution $\mathbf{u}$ is called a {\it non-topological solution} if
$$u_a(x)\to -\iy\quad\text{as $|x|\to +\iy$},\quad a=1, \cdots, r.$$

For any configuration $p_j^a$ in $\R^2$, the existence of topological solutions to \eqref{eq1-4} was proved by Yang \cite{Y1} in 1997. However,
the existence question of non-topological solutions (and mixed-type solutions, see below) is much more difficult than the one for topological solutions, and has remained open for a long time. Only recently, with the help of the classification result in \cite{LWY}, the first existence result of non-topological solutions to (\ref{eq1-4}) with $K=\mathcal{A}_2, \mathcal{B}_2$ and $\mathcal{G}_2$ have been obtained by Ao, Wei and the second author \cite{ALW1, ALW2} by the finite-dimensional reduction method through a perturbation from the Liouville system \eqref{toda}.
However, the understanding of the structure of non-topological solutions is still far from complete.

\subsection{Main result}
In this paper, we focus on the radially symmetric solutions of \eqref{eq1-4} when all the vortices coincide at the origin. We only consider the rank $2$ and competitive case, namely $K$ is a $2\times 2$ matrix satisfying \eqref{eq1-2}.
Moreover, we may assume, without loss of generality, that $\kappa=1$. Then system \eqref{eq1-4} coincides with \eqref{eq1-1}.
In particular, when $K=\mathcal{A}_2$, then \eqref{eq1-1} becomes the following $SU(3)$ Chern-Simons system
\begin{equation}\label{eq1-7}
\begin{cases}
\begin{split}
\Delta u_1+2(e^{u_1}-2e^{2u_1}+e^{u_1+u_2})-(e^{u_2}-2e^{2u_2}+e^{u_1+u_2})=4\pi N_1\dd_0\\
\Delta u_2+2(e^{u_2}-2e^{2u_2}+e^{u_1+u_2})-(e^{u_1}-2e^{2u_1}+e^{u_1+u_2})=4\pi N_2\dd_0\\
\end{split}\; \text{in $\R^2$}.
\end{cases}
\end{equation}
When $K=\mathcal{B}_2$, then system \eqref{eq1-1} becomes the following $\mathcal{B}_2$ Chern-Simons system
\begin{equation}\label{eq1-24}
\begin{cases}
\begin{split}
&\Delta u_1+2e^{u_1}-e^{u_2}-4e^{2u_1}+2e^{2u_2}=4\pi N_1\dd_0\\
&\Delta u_2+2e^{u_2}-2e^{u_1}-4e^{2u_2}+2e^{u_1+u_2}+4e^{2u_1}=4\pi N_2\dd_0\\
\end{split}\; \text{in $\R^2$}.
\end{cases}
\end{equation}

As in \cite{HL2}, in order to simplify the expression of system \eqref{eq1-1}, we consider the transformation
$$(u_1, u_2)\to \left(u_1+\ln\frac{a_{22}-a_{12}}{a_{11}a_{22}-a_{12}a_{21}},\; u_2+\ln\frac{a_{11}-a_{21}}{a_{11}a_{22}-a_{12}a_{21}}\right)$$
and let
$$(a_1, a_2)=\left(\frac{-a_{12}(a_{11}-a_{21})}{a_{11}a_{22}-a_{12}a_{21}},
\;\frac{-a_{21}(a_{22}-a_{12})}{a_{11}a_{22}-a_{12}a_{21}}\right).$$
Clearly the assumption \eqref{eq1-2} implies $a_1>0$ and $a_2>0$. Then system \eqref{eq1-1} becomes
\begin{equation}\label{problem}
\begin{cases}
\begin{split}
\Delta u_1+(1+a_1)(e^{u_1}-(1+a_1)e^{2u_1}+a_1e^{u_1+u_2})&\\
-a_1(e^{u_2}-(1+a_2)e^{2u_2}+a_2e^{u_1+u_2})=&4\pi N_1\dd_0\\
\Delta u_2+(1+a_2)(e^{u_2}-(1+a_2)e^{2u_2}+a_2e^{u_1+u_2})&\\
-a_2(e^{u_1}-(1+a_1)e^{2u_1}+a_1e^{u_1+u_2})=&4\pi N_2\dd_0\\
\end{split}\;\text{in $\R^2$}.
\end{cases}
\end{equation}
For the three types of Cartan matrices \eqref{eq1-5}, we find that
\be\label{eq1-5-1}(a_1, a_2)=\begin{cases}(1,1),\quad\text{if}\;\,K=\mathcal{A}_2,\\
(2,3),\quad\text{if}\;\,K=\mathcal{B}_2,\\
(5,9),\quad\text{if}\;\,K=\mathcal{G}_2.
\end{cases}\ee
In particular, system \eqref{eq1-7} is invariant under the above transformation.

Clearly, to study system \eqref{eq1-1}, we only need to consider system \eqref{problem}.
It is more interesting to us that, when $(a_1, a_2)$ take {\it some other special values but not} \eqref{eq1-5-1}, system \eqref{problem} also arises in some other physical models, such as the Lozano-Marqu\'{e}s-Moreno-Schaposnik model \cite{LMMS} and the Gudnason model \cite{G1, G2}. Lozano et al. \cite{LMMS} considered the bosonic sector of $\mathcal{N}=2$ supersymmetric Chern-Simons-Higgs theory when the gauge group is $U(1)\times SU(N)$ and has $N_f$ flavors of fundamental matter fields. They investigated so-called local $Z_N$ string-type solutions when $N_f=N$ and obtained a system of nonlinear differential equations (see \cite[(19)-(22)]{LMMS}) which, under a suitable change of variables and unknowns, can be transformed into \eqref{problem} with $(a_1, a_2)=(\frac{k-1}{N}, \frac{(N-1)(k-1)}{N})$ and $k>0$. If $k>1$, then $a_1, a_2>0$ and $a_1+a_2=1$. Gudnason \cite{G1,G2} considered a $\mathcal{N}=2$ supersymmetric Yang-Mills-Chern-Simons-Higgs theory with the general gauge group $G=U(1)\times G'$, where $G'$ is a non-Abelian simple Lie group represented by matrices. When the gauge group are $U(1)\times SO(2M)$ and $U(1)\times US_p(2M)$, the so-called master equations are a system of nonlinear differential equations (see \cite[(3.64)-(3.65)]{G1} or \cite[(2.1)-(2.2)]{HLTY}). Letting $M=1$ and using a suitable transformation, this system coincides with \eqref{problem} with $a_1=a_2=\frac{\bb^*-\al^*}{2\al^*}$ and $\al^*, \bb^*>0$. If $\bb^*>\al^*$, then $a_1=a_2>0$. See \cite{HL2} for these two transformations.

Therefore, it is worth for us to study system \eqref{problem} with generic $a_1, a_2>0$ rather than only \eqref{eq1-5-1}.
As in \cite{HL2}, we easily see that a solution $(u_1, u_2)$ of \eqref{problem} is a {\it topological solution} if $(u_1, u_2)\to (0, 0)$ as $|x|\to +\iy$; a {\it non-topological solution} if $(u_1, u_2)\to (-\iy, -\iy)$ as $|x|\to +\iy$; a {\it mixed-type solution} if either $(u_1, u_2)\to (\ln\frac1{1+a_1}, -\iy)$ or $(u_1, u_2)\to (-\iy, \ln\frac1{1+a_2})$ as $|x|\to +\iy$.

We postpone other mathematical details about \eqref{problem} and first state our main result. To simplify the notations, in the sequel we denote
\be\label{eq1-1-1-1}A=(1+a_1)(1+a_2)\quad\text{and}\quad B=a_1a_2.\ee
Then $A-B=1+a_1+a_2>1$. Denote $B(0, R):=\{x\in \R^2\,|\, |x|<R\}$. The main result of this paper is to prove the existence of {\it a new type of non-topological bubbling solutions}.

\begin{theorem}[=Theorem \ref{th2-1}]\label{th1-1} Assume that $a_1, a_2>0$ satisfy
\be\label{eq2-1-1}3(1+a_1)(1+a_2)-4a_1a_2>0\ee
and $N_1, N_2$ are non-negative integers satisfying
\be\label{eq2-2-1}(A-4B)(N_1+1)<2a_1(1+a_1)(N_2+1)\;\;\text{if}\;\,A-4B>0.\ee
Let $(\al_1, \al_2)$ satisfy $\al_1\ge 1$, $\al_2>1$ and
{\allowdisplaybreaks
\begin{align}
\label{eq1-20-2}&(3A-4B)\al_1
+\frac{1+a_1}{a_2}(A-2B)\al_2\nonumber\\
&\qquad\quad=AN_1+\frac{1+a_1}{a_2}AN_2
+\left(4+2\frac{1+a_1}{a_2}\right)(A-B),\\
\label{eq1-20-3}&\frac{4B-A}{A}(\al_1-1)+\frac{2a_1}{1+a_2}(\al_2-1)-(N_1+1)>0.
\end{align}
}%
Then system \eqref{problem} admits a sequence of radial non-topological bubbling solutions $(u_{1, n}, u_{2, n})$ such that $\sup_{\R^2}u_{2, n}\to-\iy$ as $n\to\iy$.
Furthermore, there exist two intersection points $R_{3,n}\gg R_{1,n}\gg 1$ of $u_{1,n}$ and $u_{2,n}$ such that:
\begin{itemize}
\item[$(i)$] $u_{1, n}\to U$ in $C^2_{loc}(B(0, R_{1,n}))$ as $n\to \iy$, where $U$ is the unique radial solution of
   \be\label{eq1-5-8}\begin{cases}\Delta U+(1+a_1)e^U-(1+a_1)^2 e^{2U}=4\pi N_1\dd_0\;\;\text{in $\R^2$},\\
U(x)=-2\ga \ln|x|+O(1)\;\;\text{as $|x|\to \iy$}\end{cases}\ee
    with
    \be\label{eq1-5-10}\ga=\frac{4B-A}{A}(\al_1-1)+\frac{2a_1}{1+a_2}(\al_2-1)+1.\ee
\item[$(ii)$] $\int_{R_{1,n}}^{R_{3,n}}re^{u_{1,n}}dr\to 0$, $\int_0^{R_{1,n}}re^{u_{2,n}}dr\to 0$, $\int_{R_{3,n}}^{\iy}re^{u_{2,n}}dr\to 0$ and
{\allowdisplaybreaks
\begin{align}& \int_{R_{1,n}}^{R_{3,n}}re^{u_{2,n}}dr\to \frac{4}{1+a_2}\left(\frac{a_2}{1+a_1}(\ga+N_1)+N_2+1\right),\nonumber\\
\label{eq1-41}&\int_{R_{3,n}}^{\iy}re^{u_{1,n}}dr\to \frac{4}{1+a_1}\left(\al_1-1\right)
\end{align}
}%
as $n\to \iy$;
\item[$(iii)$] there exist constants $\al_{1,n}>1, \al_{2,n}>1$ such that
$$u_{k,n}(r)=-2\al_{k,n}\ln r+O(1)\;\;\text{as}\;\,r\to \iy,\;\;k=1,2,$$
and $(\al_{1, n}, \al_{2, n})\to (\al_1, \al_2)$ as $n\to \iy$.
\end{itemize}
\end{theorem}

\subsection{Motivations and remarks}
In order to get a full understanding of Theorem \ref{th1-1} (such as, why we assume \eqref{eq1-20-2}), we will shortly discuss some known results.
Define a quadratic form $J: \R^2\to \R$ by
\be\label{eq1-9}J(x, y)=\frac{a_2(1+a_2)}{2}x^2+a_1a_2xy+\frac{a_1(1+a_1)}{2}y^2.\ee
Recently, in \cite{HL1, HL2}, Huang and the second author classified all radially symmetric solutions of \eqref{eq1-1}.
Among other things, they proved the following interesting result.

\medskip

\noindent{\bf Theorem A.} \cite{HL2} {\it
Let $a_1, a_2>0$ and $(u_1, u_2)\neq (0, 0)$ be a radially symmetric solution of system \eqref{problem}. Then both
$u_1<0$ and $u_2<0$ in $\R^2$, and one of the following conclusions holds.
\begin{itemize}
\item[$(i)$] $(u_1, u_2)$ is a topological solution.
\item[$(ii)$] $(u_1, u_2)$ is a mixed-type solution.
\item[$(iii)$] $(u_1, u_2)$ is a non-topological solution and there exist constants $\al_1, \al_2>1$ such that
\begin{align}\label{eq1-10}u_k(x)=-2\al_k\ln|x|+O(1)\;\;\text{as $|x|\to+\iy$},\quad k=1, 2.\end{align}
Consequently, $e^{u_1}, e^{u_2}\in L^1(\R^2)$. Moreover, $(\al_1, \al_2)$ satisfies
\be\label{eq1-11}J(\al_1-1, \al_2-1)>J(N_1+1, N_2+1).\ee
\end{itemize}
}
\medskip

The inequality \eqref{eq1-11} comes from the following Pohozaev identity (see \cite{HL2}):
\begin{equation}\label{eq1-12}
\begin{split}
&J(\al_1-1, \al_2-1)-J(N_1+1, N_2+1)\\
=&\frac{A-B}{4}\int_{0}^\iy r\left[a_2(1+a_1)e^{2u_1}+a_1(1+a_2)e^{2u_2}-2a_1a_2e^{u_1+u_2}\right]dr.
\end{split}\end{equation}
Therefore, \eqref{eq1-11} is a necessary condition for the existence of radially symmetric non-topological solutions satisfying the asymptotic condition \eqref{eq1-10}.
In view of Theorem A, it is natural to consider the following question.

\vskip 0.1in
\noindent \textbf{Question:} Fix $\al_1, \al_2>1$ satisfying \eqref{eq1-11}. Is there
a radially symmetric non-topological solution of system \eqref{problem} subject to the prescribed asymptotic condition \eqref{eq1-10}?
\vskip 0.1in

If we let $N_1=N_2=N$, $a_1=a_2$ and $u_1=u_2=u$ in \eqref{problem}, then system \eqref{problem} turns to be
the following Chern-Simons-Higgs scalar equation
\be\label{eq1-13}\Delta u+ e^u(1-e^u)=4\pi N \dd_0\quad\text{in $\R^2$}.\ee
Equation \eqref{eq1-13} is known as the $SU(2)$ Chern-Simons model for the Abelian case; see \cite{HKP, JW}. For recent developments, we refer the reader to \cite{CI, CFL, C1, C2, CKL1, SY1, SY2} and references therein.
Remark that the Pohozaev identity plays a key role in studying non-topological solutions of \eqref{eq1-13}. Let $u$ be a radial non-topological solution of \eqref{eq1-13} satisfying $u(x)=-2\al \ln|x|+O(1)$ near $\iy$. Then the Pohozaev identity implies
$$(\al-1)^2-(N+1)^2=\frac12\int_0^\iy r e^{2u}dr>0,$$
which implies $\al>N+2$. In 2002,
Chan, Fu and the second author \cite{CFL} proved that the inequality $\al>N+2$ is also a sufficient condition for the existence of radial non-topological solutions
satisfying $u(x)=-2\al \ln|x|+O(1)$ near $\iy$.
However, as pointed out in \cite{CL, CKL2}, this might not hold for system \eqref{problem}. The reason is following: there might be a sequence of solutions $(u_{1, n}, u_{2, n})$ such that only one component blows up, but the other one does not, i.e. the so-called phenomena of {\it partial blowup}; see Theorem C for instance. As a result, only one of the $L^1$ norms of $e^{2u_{1, n}}$ and $e^{2u_{2, n}}$ tends to $0$ as $n\to \iy$, which implies that the quantity $J(\al_1-1, \al_2-1)-J(N_1+1, N_2+1)$ might not converge to $0$, namely it has a gap. Therefore, the inequality \eqref{eq1-11} might {\it not} be a sufficient condition for the existence of radial non-topological solutions satisfying \eqref{eq1-10}.

In a previous work \cite{CL}, we found a sufficient condition for the above question. As in \cite{CL}, we define
\begin{equation}\label{eq1-16}
\Om:=\left\{(\al_1, \al_2) \;|\; \al_1, \al_2>1, \; J(\al_1-1, \al_2-1)>J(N_1+1, N_2+1)\right\},
\end{equation}
and
\be\label{eq1-17}S:=\left\{(\al_1, \al_2)\;|\; \al_1, \al_2>0\;\;\text{and $(\al_1, \al_2)$ satisfies \eqref{eq1-18}}-\eqref{eq1-21}\right\},\ee
where
{\allowdisplaybreaks
\begin{align}
\label{eq1-18}&(A-2B)\al_2-a_2(1+a_2)\al_1
<a_2(1+a_2)N_1+AN_2+2(A-B),\\
\label{eq1-19}&(A-2B)\al_1-a_1(1+a_1)\al_2
<a_1(1+a_1)N_2+AN_1+2(A-B),\\
\label{eq1-20}&(3A-4B)\al_1
+\tfrac{1+a_1}{a_2}(A-2B)\al_2\nonumber\\
&\quad\qquad>AN_1+\tfrac{1+a_1}{a_2}AN_2
+\left(4+2\tfrac{1+a_1}{a_2}\right)(A-B),\\
\label{eq1-21}&(3A-4B)\al_2
+\tfrac{1+a_2}{a_1}(A-2B)\al_1\nonumber\\
&\quad\qquad>AN_2+\tfrac{1+a_2}{a_1}AN_1
+\left(4+2\tfrac{1+a_2}{a_1}\right)(A-B).
\end{align}
}%

\medskip
\noindent{\bf Theorem B.} \cite{CL} {\it
Assume that $N_1, N_2$ are non-negative integers and $a_1, a_2>0$ satisfy
\be\label{eq1-22}(1+a_1)(1+a_2)>\left(6-2\sqrt{5}\right)a_1a_2.\ee
Let $\Om$ and $S$ be defined in \eqref{eq1-16}-\eqref{eq1-17}. Then $S\cap \Om\neq \emptyset$, and for any fixed $(\al_1, \al_2)\in S\cap\Om$, system \eqref{problem} admits a radially symmetric non-topological solution $(u_1, u_2)$ satisfying the prescribed asymptotic condition \eqref{eq1-10}.}
\medskip

Remark that $S\cap\Om\neq\emptyset$ if and only if $(a_1, a_2)$ satisfies \eqref{eq1-22}, i.e. \eqref{eq1-22} is a necessary condition for Theorem B. For example, Theorem B can be applied to the $SU(3)$ system \eqref{eq1-7} and the $\mathcal{B}_2$ system \eqref{eq1-24}. The counterpart of Theorem B for the $SU(3)$ system \eqref{eq1-7} was firstly obtained by Choe, Kim and the second author \cite{CKL2}, and Theorem B is a generalization of their result to the generic system \eqref{problem}. Applying Theorem B to the $\mathcal{B}_2$ case, we conclude that if
\be\label{eq1-1-1}
(\alpha _{1},\alpha _{2})\in S\cap \Omega =\left \{ (\alpha _{1},\alpha
_{2})\left \vert
\begin{array}{l}
\alpha _{1}>N_{1}+N_{2}+3 \\
\alpha _{2}>2N_{1}+N_{2}+4%
\end{array}%
\right. \right \},\ee
then \eqref{eq1-24} has a radial non-topological solution satisfying \eqref{eq1-10}.

We proved Theorem B via the Leray-Schauder degree theory. To do this, we proved a uniform boundedness result for radial solutions satisfying \eqref{eq1-10} whenever $(\al_1, \al_2)\in S\cap\Om$.
Then a natural question is {\it whether the set $S\cap \Om$ is the optimal range of $(\al_1, \al_2)$ for the existence of radial solutions satisfying \eqref{eq1-10}}. This question has not been settled yet (Theorem \ref{th1-1} gives {\it a negative answer} for the $\mathcal{B}_2$ case; see Remark \ref{remark1-2} below). However, in the same paper \cite{CL}, we also proved the existence of partially bubbling solutions along some part of $\partial(S\cap \Om)$.

\medskip

\noindent{\bf Theorem C.} \cite{CL} {\it Assume that $N_1, N_2$ are non-negative integers and $a_1, a_2>0$ satisfy
\be\label{eq1-26-1}(1+a_1)(1+a_2)> 2a_1a_2.\ee
Let $(\al_1, \al_2)\in \Om$ satisfy $\al_1\neq \al_2$ and
\begin{equation}\label{eq1-26}
(A-2B)\al_2-a_2(1+a_2)\al_1
=a_2(1+a_2)N_1+AN_2+2(A-B).
\end{equation}
Then system \eqref{problem} admits a sequence of radial non-topological bubbling solutions $(u_{1, n}, u_{2, n})$ such that $\sup_{\R^2}u_{2, n}\to-\iy$ as $n\to\iy$. Furthermore, there exists a intersection point $R_{1,n}\gg 1$ of $u_{1,n}$ and $u_{2,n}$ such that:
\begin{itemize}
\item[$(i)$] $u_{1, n}\to U$ in $C^2_{loc}(B(0, R_{1,n}))$, where $U$ is the unique radial solution of (\ref{eq1-5-8}) with $\ga=\al_1+\frac{2a_1}{1+a_2}(\al_2-1)$. Besides,
\be\label{eq1-5-9} \lim_{n\to\iy}\int_{R_{1,n}}^{\iy}re^{u_{1,n}}dr=0.
\ee
\item[$(ii)$] there exists $(\al_{1, n}, \al_{2, n})\in\Om$ such that
$$u_{k,n}(r)=-2\al_{k,n}\ln r+O(1)\;\;\text{as}\;\,r\to \iy,\;\;k=1,2,$$
and $(\al_{1, n}, \al_{2, n})\to (\al_1, \al_2)$ as $n\to\iy$.
\end{itemize}
}
\medskip

Theorem C proves the existence of bubbling solutions along the boundary of \eqref{eq1-18}.
For these bubbling solutions, only the second component blows up. We call this type of bubbling solutions of {\it type I}. Inspired by Theorem B, there might exist another type of bubbling solutions along the boundary of \eqref{eq1-20} (or equivalently \eqref{eq1-21}), which we call of {\it type II}. But for type II, the estimate \eqref{eq1-5-9} no longer holds, which means that {\it both components of bubbling solutions blow up at infinity}, namely the asymptotics of type II are more complicated.

Therefore, Theorem \ref{th1-1} gives precisely the existence of bubbling solutions of type II. We conclude this section by some further comments about Theorem \ref{th1-1}.

\br
For the $SU(3)$ case, we have $(a_1, a_2)=(1,1)$. Then it is easy to check that the range of $(\al_1, \al_2)$ given in Theorem \ref{th1-1} is exactly
\be\label{eq1-2-1}2\al_1+\al_2=N_1+2N_2+6\quad\text{and}\quad 1\le\al_1<N_2+2.\ee
We remark that the counterpart of Theorem \ref{th1-1} for the $SU(3)$ system \eqref{eq1-7} was firstly proved by Choe, Kim and the second author \cite{CKL3} under  the following assumption
\be\label{eq11-2-1}2\al_1+\al_2=N_1+2N_2+6\quad\text{and}\quad 1<\al_1<N_2+2,\ee
where the assumption $\al_1>1$ plays a crucial role in their proof.
Theorem \ref{th1-1} improves their result on two aspects. First, for the $SU(3)$ case, Theorem \ref{th1-1} covers the special case $\al_1=1$ (note that $(\al_1,\al_2)=(1, N_1+2N_2+4)\in \partial (S\cap \Om)$) which was not considered in \cite{CKL3}. Remark that the case $\al_1=1$ is different from the case $\al_1>1$. Indeed, we can see from \eqref{eq1-41} that, if $\al_1>1$, then \eqref{eq1-5-9} no longer holds, namely the bubbling solutions are of type II. However, the case $\al_1=1$ satisfies \eqref{eq1-5-9}, namely the bubbling solutions are of type I just as in Theorem C. This phenomena is reasonable, because the intersection point of line \eqref{eq1-26} with line \eqref{eq1-20-2}, which exists provided $A>2B$, is exactly
$$(\al_1, \al_2)=\left(1,\; \tfrac{a_2(1+a_2)}{A-2B}(N_1+1)+\tfrac{A}{A-2B}(N_2+1)+1\right).$$
Thus the case $\al_1=1$ can be seen as a critical case that connects bubbling solutions of type I with bubbling solutions of type II.
Observe that $\al_1>1$ was assumed in Theorem C, so our study of the case $\al_1=1$ is also a complement of Theorem C. Second, Theorem \ref{th1-1} generalizes their result to the generic system \eqref{problem}.
Theorem \ref{th1-1} indicates that, there exist bubbling solutions of type II along the boundary of \eqref{eq1-20}. This fact, together with Theorem C, shows that the set $S\cap\Om$ is an optimal range\footnote{Here we mean that the a priori estimates can not hold for any open connected set containing $S\cap\Om$ as a proper subset. In other words, if there is another connected range $\tilde{S}$ for the existence of solutions by the degree theory, then $\tilde{S}\cap \overline{S\cap\Om}=\emptyset$.} in view of the degree theory.
\er

\br\label{remark1-2}
For the $SU(3)$ system \eqref{eq1-7}, we still do not know whether the set $S\cap \Om$ is the optimal range for the existence of non-topological solutions, but we strongly believe so in view of Theorems B, C and \ref{th1-1}; see \cite{CKL2, CKL3}. However, the generic system \eqref{problem} is more involved than the $SU(3)$ system \eqref{eq1-7}. One example is the $\mathcal{G}_2$ case where $(a_1, a_2)=(5, 9)$ and so $3A-4B=0$. Therefore, none of Theorems B, C and \ref{th1-1} can be applied to the $\mathcal{G}_2$ case, and understanding the non-topological solution structure for the $\mathcal{G}_2$ case remains open.  Another example is the $\mathcal{B}_2$ case, where $(a_1, a_2)=(2, 3)$ and so $A-2B=0$. Then it is easy to check that the range of $(\al_1, \al_2)$ given in Theorem \ref{th1-1} is exactly
\be\label{eq1-2-2}\al_1=N_1+N_2+3\quad\text{and}\quad \al_2>1.\ee
From here, we conclude that the set $S\cap\Om$ (see \eqref{eq1-1-1}) given in Theorem B, which is optimal in view of the degree theory, is not the optimal range for the existence of non-topological solutions to the $\mathcal{B}_2$ system \eqref{eq1-24}.
\er

\br
Theorem B can not be applied to the case $A\le (6-2\sqrt{5})B$. Therefore, Theorem \ref{th1-1} also gives the first existence result of radial non-topological solutions for the case $\frac{4}{3}B<A\le (6-2\sqrt{5})B$.
\er

\br
Clearly, by Theorem A-$(iii)$, assumptions $\al_1\ge 1, \al_2>1$ are necessary for Theorem \ref{th1-1}.
As mentioned before, $\ga>N_1+2$ is a necessary and sufficient condition for the existence of radial solutions for \eqref{eq1-5-8}. Therefore, \eqref{eq1-5-10} indicates that \eqref{eq1-20-3} is a necessary condition for Theorem \ref{th1-1}. In fact,  \eqref{eq1-20-3} is also needed to guarantee that $(\al_1, \al_2)$ satisfies inequality \eqref{eq1-11} (see Lemma \ref{lemma2-2} below), which is obviously necessary by the Pohozaev identity. On the other hand,
assumptions \eqref{eq2-1-1}-\eqref{eq2-2-1} are also necessary conditions for Theorem \ref{th1-1}, because they are the necessary and sufficient condition to guarantee $\{(\al_1, \al_2)\,|\, \al_1\ge 1, \al_2>1\;\text{and satisfy}\;\eqref{eq1-20-2}-\eqref{eq1-20-3}\}\neq \emptyset$; see Lemma \ref{lemma2-1} below.
\er

Theorem \ref{th1-1} will be proved via the shooting method in Section 2.

\vskip0.10in
\section{Construction via the shooting method}
\renewcommand{\theequation}{2.\arabic{equation}}

In this section, we will prove Theorem \ref{th1-1} by constructing bubbling solutions via the shooting method. In the sequel, we assume that $a_1, a_2>0$ satisfy
\be\label{eq2-1}3(1+a_1)(1+a_2)-4a_1a_2>0.\ee
Recall the notations $A, B$ in \eqref{eq1-1-1-1}. Assume that $N_1, N_2$ are non-negative integers satisfying
\be\label{eq2-2}(A-4B)(N_1+1)<2a_1(1+a_1)(N_2+1)\;\;\text{if}\;\,A-4B>0.\ee
Define
\be\label{eq2-0-1}
\Sigma:=\{(\al_1, \al_2)\,|\, \al_1\ge 1, \;\al_2>1,\; g(\al_1, \al_2)=0,\; h(\al_1, \al_2)>0\},
\ee
where
{\allowdisplaybreaks
\begin{align}
\label{eq2-3}g(\al_1, \al_2):=&(3A-4B)\al_1
+\frac{1+a_1}{a_2}(A-2B)\al_2-AN_1\nonumber\\
&-\frac{1+a_1}{a_2}AN_2
-\left(4+2\frac{1+a_1}{a_2}\right)(A-B),\\
\label{eq2-3-1}h(\al_1, \al_2):=&\frac{4B-A}{A}(\al_1-1)+\frac{2a_1}{1+a_2}(\al_2-1)-(N_1+1).
\end{align}
}%

\bl\label{lemma2-1} $\Sigma\neq\emptyset$ if and only if \eqref{eq2-1}-\eqref{eq2-2} hold.\el

\begin{proof}
Denote $\tilde{N}_k=N_k+1$ for $k=1, 2$. Clearly $\Sigma\neq\emptyset$ is equivalent to
$$\tilde{\Sigma}:=\left\{(x, y)\,|\, x\ge 0,\, y>0,\, \tilde{g}(x, y)=0,\, \tilde{h}(x, y)>0\right\}\neq\emptyset,$$
where
{\allowdisplaybreaks
\begin{align*}
\tilde{g}(x, y):=&(3A-4B)x
+\frac{1+a_1}{a_2}(A-2B)y-A\tilde{N}_1-\frac{1+a_1}{a_2}A\tilde{N}_2,\nonumber\\
\tilde{h}(x, y):=&\frac{4B-A}{A}x+\frac{2a_1}{1+a_2}y-\tilde{N}_1.
\end{align*}
}%
If $3A-4B\le 0$, then $\{(x, y) \,|\, x\ge 0, y>0,\, \tilde{g}(x, y)=0\}=\emptyset$. Therefore, \eqref{eq2-1} is a necessary condition to guarantee $\tilde{\Sigma}\neq \emptyset$.
In the following, we always assume that \eqref{eq2-1} holds.
Then it is trivial to see that $\tilde{\Sigma}\neq\emptyset$ in the case $A-2B\le 0$. Consider the remaining case $A-2B>0$. Observe that the intersection point of $\tilde{g}(x, y)=0$ with the $y$-axis is $(0, \frac{a_2(1+a_2)\tilde{N}_1+A\tilde{N}_2}{A-2B})$.
If $A-4B\le 0$, a direct computation shows that
\be\label{eq2-0}\tilde{h}\left(0,\; \tfrac{a_2(1+a_2)\tilde{N}_1+A\tilde{N}_2}{A-2B}\right)>0\ee
holds automatically, which implies $\tilde{\Sigma}\neq 0$. If $A-4B>0$, it is easy to see that $\tilde{\Sigma}\neq\emptyset$ if and only if \eqref{eq2-0} holds, which is just equivalent to \eqref{eq2-2}. This completes the proof.
\end{proof}

In the sequel, we
fix any $(\al_1, \al_2)\in \Sigma$. We will prove the existence of bubbling solutions near $(\al_1, \al_2)$ just as stated in Theorem \ref{th1-1}.

Inspired by the blowup analysis in our previous work \cite{CL}, we define
\be\label{eq2-4}\ga:=\frac{4B-A}{A}(\al_1-1)+\frac{2a_1}{1+a_2}(\al_2-1)+1.\ee
Then $h(\al_1, \al_2)>0$ gives
\be\label{eq2-4-1}\ga>N_1+2.\ee
Clearly, $g(\al_1, \al_2)=0$ and \eqref{eq2-4} yield
{\allowdisplaybreaks
\begin{align}
\label{eq2-5}
&\al_1=-\frac{A-2B}{A}(\ga-1)+\frac{2B}{A}(N_1+1)+\frac{2a_1}{1+a_2}(N_2+1)+1,\\
\label{eq2-6}&\al_2=\frac{a_2}{1+a_1}\frac{3A-4B}{A}(\ga-1)
+\frac{a_2}{1+a_1}\frac{A-4B}{A}(N_1+1)\nonumber\\
&\qquad+\frac{A-4B}{A}(N_2+1)+1.
\end{align}
}%
By $\al_1\ge 1$ we obtain
\be\label{eq2-5-1}\ga\le 1+\frac{2B}{A-2B}(N_1+1)+\frac{2a_1(1+a_1)}{A-2B}(N_2+1)
\;\;\text{if}\;\,A>2B.\ee
As in Lemma \ref{lemma2-1}, for convenience, we always denote
\be\label{eq2-7}\tilde{\ga}=\ga-1,\;\;\tilde{\al}_k=\al_k-1
\;\;\text{and}\;\;\tilde{N}_k=N_k+1,\;\;k=1,2.\ee

\bl\label{lemma2-2} $h(\al_1, \al_2)>0$ implies $J(\al_1-1, \al_2-1)>J(N_1+1, N_2+1)$.\el

\begin{proof}
By the definition \eqref{eq1-9} of $J$, a direct computation shows
{\allowdisplaybreaks
\begin{align}\label{eq2-8}
J(x, y)&=J(-x, -y)\nonumber\\
&=J\left(x,\, -\frac{2a_2}{1+a_1}x-y\right)=J\left(-x,\, \frac{2a_2}{1+a_1}x+y\right)\nonumber\\
&=J\left(-x-\frac{2a_1}{1+a_2}y,\,y\right)=J\left(x+\frac{2a_1}{1+a_2}y,\,-y\right).
\end{align}
}%
Since \eqref{eq2-5}-\eqref{eq2-6} give
$$\tilde{\al}_1=\frac{2B-A}{A}\tilde{\ga}+
\frac{2B}{A}\tilde{N}_1+\frac{2a_1}{1+a_2}\tilde{N}_2,$$
$$\tilde{\al}_2=\frac{a_2}{1+a_1}(\tilde{\ga}+\tilde{N}_1)
+\tilde{N}_2-\frac{2a_2}{1+a_1}\tilde{\al}_1,$$
we can derive
{\allowdisplaybreaks
\begin{align*}
J(\tilde{\al}_1, \tilde{\al}_2)&=J\left(\tilde{\al}_1, \; \frac{a_2}{1+a_1}(\tilde{\ga}+\tilde{N}_1)
+\tilde{N}_2-\frac{2a_2}{1+a_1}\tilde{\al}_1\right)\\
=&J\left(\tilde{\al}_1, -\frac{a_2}{1+a_1}(\tilde{\ga}+\tilde{N}_1)
-\tilde{N}_2\right)\\
=&J\left(\frac{2a_1}{1+a_2}\left(\frac{a_2}{1+a_1}(\tilde{\ga}+\tilde{N}_1)
+\tilde{N}_2\right)-\tilde{\ga}, \; -\frac{a_2}{1+a_1}(\tilde{\ga}+\tilde{N}_1)
-\tilde{N}_2\right)\\
=&J\left(-\tilde{\ga},\;\frac{a_2}{1+a_1}(\tilde{\ga}+\tilde{N}_1)
+\tilde{N}_2\right)\\
=&J(\tilde{N}_1, \tilde{N_2})+\frac{a_2(1+a_1+a_2)}{2(1+a_1)}\left(\tilde{\ga}^2-\tilde{N}_1^2\right).
\end{align*}
}%
By $\ga>N_1+2$, we conclude $J(\al_1-1, \al_2-1)>J(N_1+1, N_2+1)$.
\end{proof}

Since $\gamma>N_1+2$, by \cite[Theorem 2.1]{CFL}, there is a unique radially symmetric solution $U$ of the Chern-Simons-Higgs equation
\be\label{eq5-8}\begin{cases}\Delta U+(1+a_1)e^U-(1+a_1)^2 e^{2U}=4\pi N_1\dd_0\;\;\text{in $\R^2$},\\
U(x)=-2\ga \ln|x|+O(1)\;\;\text{as $|x|\to \iy$}.\end{cases}\ee
Moreover, $U<-\ln (1+a_1)$ in $\R^2$ and
{\allowdisplaybreaks
\begin{align}&\int_{0}^\iy r\left[(1+a_1)e^U-(1+a_1)^2 e^{2U}\right]dr=2(\ga+N_1),\nonumber\\
\label{eq5-9} &\lim_{r\to\iy} \left[r^2e^{U(r)}+|rU'(r)+2\ga|\right]=0.
\end{align}
}%
Let $V(|x|)=V(x):=U(x)-2N_1\ln |x|$, then $V(0):=\lim_{r\to 0}V(r)$ is well defined; see \cite{CFL}.

To use the shooting method, we consider an initial value problem of system \eqref{problem} in a radial variable. Denote
$$F_k(r):=(1+a_k)e^{2u_k(r)}-e^{u_k(r)}-a_k e^{u_1(r)+u_2(r)},\;\;k=1, 2,$$
for convenience.
Clearly
\be\label{eq5-10-1}|F_k|\le (1+a_k)e^{u_k}\le 1+a_k\;\;\text{whenever}\;\; u_1, u_2\le 0,\ee
and it is easy to check that
\be\label{eq5-10-2}
\begin{cases}
\begin{split}
&F_k<-\tfrac{1}{2}e^{u_k}<0\;\,\text{if}\;\,u_k<-\ln2(1+a_1+a_2),\\
&F_{3-k}<F_k<0\;\,\text{if}\;\,u_k<u_{3-k}<-\ln2(1+a_1+a_2),
\end{split}\;\, k=1, 2.
\end{cases}\ee
We study the following initial value problem
\begin{equation}\label{eq5-10}
\begin{cases}
u_1''(r)+\frac{1}{r}u_1'(r)=(1+a_1)F_1(r)-a_1 F_2(r),\;\;r>0,\\
u_2''(r)+\frac{1}{r}u_2'(r)=(1+a_2)F_2(r)-a_2 F_1(r),\;\; r>0,\\
u_1(r)=2N_1\ln r+V(0)+o(1),\;\; r\to 0,\\
u_2(r)=2N_2\ln r+\ln \e+o(1),\;\; r\to 0,
\end{cases}
\end{equation}
where $\e\in (0, 1)$. Clearly, the solution of \eqref{eq5-10} depends on $\e$ and we denote it by $(u_{1, \e}, u_{2, \e})$. Consequently, $F_k(r)=F_{k,\e}(r)$ also depends on $\e$. For the sake of convenience, when no confusion arise, we will omit the subscript $\e$.
The main result of this section is following, and Theorem \ref{th1-1} is a direct corollary.

\begin{theorem}\label{th2-1}
Assume $\eqref{eq2-1}-\eqref{eq2-2}$ and fix any $(\al_1, \al_2)\in \Sigma$. Then there exists sufficiently small $\e_0>0$ such that for any $\e<\e_0$, system \eqref{eq5-10} has an entire solution $(u_{1, \e}, u_{2, \e})$. Furthermore, there exist two intersection points $R_{3,\e}\gg R_{1,\e}\gg 1$ of $u_{1,\e}$ and $u_{2,\e}$ such that:
\begin{itemize}
\item[$(i)$] $u_{1, \e}\to U$ in $C^2_{loc}(B(0, R_{1,\e}))$ and $\sup\limits_{\R^2} u_{2, \e}\to-\iy$ as $\e\to 0$;
\item[$(ii)$] $\int_{R_{1,\e}}^{R_{3,\e}}re^{u_{1,\e}}dr\to 0$, $\int_0^{R_{1,\e}}re^{u_{2,\e}}dr\to 0$, $\int_{R_{3,\e}}^{\iy}re^{u_{2,\e}}dr\to 0$ and
{\allowdisplaybreaks
\begin{align*}& \int_{R_{1,\e}}^{R_{3,\e}}re^{u_{2,\e}}dr\to \frac{2}{1+a_2}\left(\frac{2a_2}{1+a_1}(\tilde{\ga}+\tilde{N}_1)+2\tilde{N}_2\right),\\
&\int_{R_{3,\e}}^{\iy}re^{u_{1,\e}}dr\to \frac{4}{1+a_1}\left(\al_1-1\right)
\end{align*}
}%
as $\e\to 0$;
\item[$(iii)$] there exists $(\al_{1,\e}, \al_{2,\e})\in \Om$ such that
$$u_{k,\e}(r)=-2\al_{k,\e}\ln r+O(1)\;\;\text{as}\;\,r\to \iy,\;\;k=1,2,$$
and $(\al_{1, \e}, \al_{1, \e})\to (\al_1, \al_2)$ as $\e\to 0$.
\end{itemize}

\end{theorem}

In the rest of this section, we prove Theorem \ref{th2-1}, which is quite long and delicate, and we divide it into several lemmas. The basic strategy is similar to that in \cite{CKL3} where the $SU(3)$ case was studied. Since the solution $(u_{1, \e}, u_{2, \e})$ of the initial value problem \eqref{eq5-10} exists locally, the key point is to prove that $(u_{1, \e}, u_{2, \e})$ exists globally for $r\in (0,+\iy)$ (i.e., does not blow up at finite $r$) provided that $\e>0$ is sufficiently small. This is the most difficult part of the proof. To overcome this difficulty, we need to carry on a delicate analysis of the asymptotic behavior of $(u_{1, \e}, u_{2, \e})$ as $r\to+\iy$. For example, we need to understand what happens if $u_{1, \e}$ and $u_{2, \e}$ intersect and how many times they intersect. The main tool is the well-known Pohozaev identity together with the blow up analysis, which makes our argument rather technical and involved.

Moreover, as pointed out
in \cite{CL,HT, HL2}, the generic situation poses new analytical difficulties compared to the $SU(3)$ case. For example, in \cite{CKL3} they used many helpful inequalities, which hold in the $SU(3)$ case because of $(a_1, a_2)=(1,1)$ but can not hold for general $a_1, a_2>0$. This requires us to develop generic ideas to avoid using this kind of inequalities. It is interesting that our general idea turns out to be somewhat simpler. More importantly, our general idea also works for the critical case $\al_1=1$. This is reasonable in view of mathematics. Roughly speaking, people usually exploit a more special method when the problem is more special. When the problem is more general, people need to develop a more generic method, the idea of which might be more natural and hence simpler.

In the sequel, we denote positive constants independent of $\e$ (possibly different in different places) by $C, C_0, C_1,\cdots$.
Let $\dd$ be a constant such that
\be\label{eq3-1}0<\dd<\frac{\min\{1, a_1, a_2, B\}\cdot\min\{1, \hat{\al}_1, \tilde{\al}_2, \ga-(N_1+2)\}}{100(1+a_1+a_2)^4},\ee
where $\hat{\al}_1:=\tilde{\al}_1=\al_1-1$ if $\al_1>1$ and $\hat{\al}_1:=1$ if $\al_1=1$.
Then by \eqref{eq5-8}-\eqref{eq5-9}, we can fix a constant $R_0>1$ large enough such that
{\allowdisplaybreaks
\begin{align}
\label{eq3-2}&R_0^2 e^{U(R_0)}+\left|R_0 U'(R_0)+2\ga\right|<\dd^3,\;\;
\int_{R_0}^\iy re^{U(r)}dr<\dd^3,\nonumber\\
&\left|\int_{0}^{R_0} r\left[e^U-(1+a_1) e^{2U}\right]dr-\frac{2(\ga+N_1)}{1+a_1}\right|<\dd^3.
\end{align}
}%
Repeating the argument of \cite[Lemma 5.2]{CL}, we can prove the existence of small $\e_1>0$ such that for each $\e\in (0, \e_1)$, problem \eqref{eq5-10} admits a solution $(u_{1, \e}, u_{2, \e})$ on $[0, R_0]$ which satisfies:
\begin{itemize}
\item[$(1)$] Both $u_{1,\e}<0$ and $u_{2,\e}<0$ on $[0, R_0]$, $u_{2,\e}(R_0)<u_{1,\e}(R_0)<2\ln\dd$ and $R_0^2e^{u_{1, \e}(R_0)}<\dd$.
\item[$(2)$] $|R_0u_{1, \e}'(R_0)+2\ga|<\dd^2$ and $|R_0u_{2, \e}'(R_0)-\frac{2a_2}{1+a_1}(\ga+N_1)-2N_2|<\dd^2$.
\item[$(3)$] $|u_{1, \e}-U|\to 0$ and $u_{2, \e}\to-\iy$ uniformly on $[0, R_0]$ as $\e\to 0$.
\end{itemize}
Furthermore, by following the argument of \cite[Lemma 5.3]{CL}, for each $\e\in (0, \e_1)$, there exists $R_1=R_{1, \e}>R_0$ such that
\begin{itemize}
\item[$(4)$] $u_{1, \e}(R_{1, \e})=u_{2, \e}(R_{1, \e})$, $u_{2,\e}<u_{1,\e}<2\ln\dd$ on $[R_0, R_{1,\e})$ and
\begin{equation}\label{eq4-0}\begin{cases}\begin{split}&|ru_{1,\e}'(r)+2\ga|<\dd,\\ &\left|ru_{2,\e}'(r)-\tfrac{2a_2}{1+a_1}(\ga+N_1)-2N_2\right|<\dd,\\
\end{split}\quad\forall\,r\in [R_0, R_{1,\e}].
\end{cases}\end{equation}
\end{itemize}
In particular, together with \eqref{eq3-1}, we have
\be\label{eq3-23} r u_{1,\e}'(r)<-2\ga+\dd<-4\;\;\text{for any}\;\,r\in [R_0, R_{1,\e}].\ee

\bl\label{lemma2-3} Let $\e\to 0$, then $R_{1,\e}\to \iy$, $R_{1,\e}^2e^{u_k(R_{1,\e})}\to 0$ for $k=1, 2$ and $\int_0^{R_{1,\e}}re^{u_{2,\e}}dr\to 0$. Furthermore,
$$R_{1,\e}u_{1,\e}'(R_{1,\e})=-2\ga+o(1),$$ $$R_{1,\e}u_{2,\e}'(R_{1,\e})=\frac{2a_2}{1+a_1}(\ga+N_1)+2N_2+o(1).$$\el

\begin{proof}
By property $(3)$ above, $u_1(R_0)-u_{2}(R_0)\to \iy$ as $\e\to 0$. Since \eqref{eq4-0} gives $|ru_k'(r)|\le C$ on $[R_0, R_1]$, we easily obtain $R_1\to \iy$ as $\e\to 0$.
Then it follows from \eqref{eq3-23} and property $(1)$ that
$$R_1^2e^{u_2(R_1)}=R_1^2e^{u_1(R_1)}<\frac{1}{R_1^2}R_0^4e^{u_1(R_0)}\to 0\;\; \text{as $\e\to 0$}.$$
By \eqref{eq4-0} again, we know that $u_2'(r)>0$ for $r\in [R_0, R_1]$. Then by property $(3)$, we have
$$\int_0^{R_1} re^{u_2(r)}dr=o(1)+\int_{R_0}^{R_1} re^{u_2(r)}dr\le o(1)+R_1^2e^{u_2(R_1)}=o(1)$$
as $\e\to 0$. Moreover, $e^{u_2}\to 0$ uniformly on $[0, R_1]$, which implies
$$ \big|(1+a_1)e^{u_1(r)+u_2(r)}+F_2(r)\big|\le2(A-B)e^{u_2(r)}\to 0\;\;\text{uniformly on $[0, R_1]$}.$$
Recall that
$$u_1''(r)+\tfrac{1}{r}u_1'(r)=\left[(1+a_1)^2 e^{2u_1}-(1+a_1)e^{u_1}\right]-a_1
\left[(1+a_1)e^{u_1+u_2}+F_2\right].$$
By the standard continuous dependence on data in the ODE theory and $R_1\to \iy$ as $\e\to 0$, we conclude  that $|u_1-U|\to 0$ uniformly on any compact subset $K\subset\subset[0, \iy)$. This, together with $e^{u_1(r)}\le r^{-4}R_0^4 e^{u_1(R_0)}$ for $r\in [R_0, R_1]$, easily yields
\begin{align*}\lim_{\e\to 0}\int_{0}^{R_1}r\left[e^{u_1}-(1+a_1) e^{2u_1}\right]dr&=\int_0^\iy r\left[e^U-(1+a_1) e^{2U}\right]dr\\
&=\frac{2(\ga+N_1)}{1+a_1}.
\end{align*}
Moreover,
$$\lim_{\e\to 0}\int_{0}^{R_1}r(e^{2u_2}+e^{u_1+u_2})dr=0.$$
Consequently, by integrating \eqref{eq5-10} over $(0, R_1)$, we obtain
{\allowdisplaybreaks
\begin{align*}
R_1u_1'(R_1)&=2N_1-(1+a_1)\int_{0}^{R_1}r\left[e^{u_1}-(1+a_1) e^{2u_1}\right]dr+o(1)\\&=-2\ga+o(1),\\
R_1u_2'(R_1)&=2N_2+a_2\int_{0}^{R_1}r\left[e^{u_1}-(1+a_1) e^{2u_1}\right]dr+o(1)\\
&=\frac{2a_2}{1+a_1}(\ga+N_1)+2N_2+o(1)
\end{align*}
}%
as $\e\to 0$. This completes the proof.
\end{proof}

For each $\e\in (0, \e_1)$, we define
$$R^*=R^*_{\e}:=\sup\{r>R_{1,\e}\,|\;\text{$u_{1,\e}<0$, $u_{2,\e}<0$ on $[0, r)$}\}.$$
Our final goal is to prove $R_\e^*=\iy$ provided $\e>0$ sufficiently small. To this goal, we define
$$R_2=R_{2,\e}:=\sup\{r\in (R_{1,\e}, R_\e^*)\,|\, u_{2,\e}'>0\;\text{on}\; [R_{1,\e}, r)\}.$$

First, we recall the following Pohozaev identity (c.f. \cite[Lemma 7.2]{HL2} or \cite[Lemma 2.2]{CL})
{\allowdisplaybreaks
\begin{align}\label{pohozaev}
&\frac{d}{dr}\bigg\{J(ru_1'(r)+2, ru_2'(r)+2)+
(1+a_1+a_2)r^2\bigg[a_2e^{u_1}+a_1e^{u_2}\nonumber\\
&\qquad-\frac{a_2(1+a_1)}{2}e^{2u_1}
-\frac{a_1(1+a_2)}{2}e^{2u_2}+a_1a_2e^{u_1+u_2}\bigg]\bigg\}\\
=&(1+a_1+a_2)r\left[a_2(1+a_1)e^{2u_1}+a_1(1+a_2)e^{2u_2}-2a_1a_2 e^{u_1+u_2}\right].\nonumber
\end{align}
}%
where $J$ is defined in \eqref{eq1-9}.

\bl\label{lemma2-4}There exists a small $\e_2\in (0, \e_1)$ such that for each $\e\in (0, \e_2)$, $u_{1,\e}<u_{2,\e}$ on $(R_{1,\e}, R_{2, \e})$.\el

\begin{proof} Assume by contradiction that there exist a sequence $\e_n\downarrow 0$ and $r_n\in (R_{1,\e_n}, R_{2,\e_n})$ such that $u_{1,\e_n}(r_n)=u_{2,\e_n}(r_n)$ and $u_{1,\e_n}<u_{2,\e_n}$ on $(R_{1,\e_n}, r_n)$. Clearly $u_{1,\e_n}'(r_n)\ge u_{2, \e_n}'(r_n)>0$ for all $n$. We will omit the subscript $\e_n$ in the following argument for convenience. We divide the proof into seven steps.

{\it Step 1.} We claim that $u_2(r_n)\to -\iy$ as $\e_n\downarrow 0$.

If not, we may assume, up to a subsequence, that $u_2(r_n)\ge c_0$ for some constant $c_0< \ln\dd$. Recall $u_2(R_1)\to -\iy$. For $n$ large, there exist $b_n, d_n\in (R_1, r_n)$ such that $b_n<d_n$, $u_2(d_n)=c_0-1$, $u_2(b_n)=c_0-2$, $u_2\le c_0-1$ on $[R_1, d_n]$ and $u_2\ge c_0-2$ on $[b_n, d_n]$. Consequently, $u_1<u_2\le c_0-1<\ln\dd$ on $(R_1, d_n]$ and so \eqref{eq5-10-2} gives $F_2<F_1<0$ on $(R_1, d_n]$, which implies
\be\label{eq3-9}(ru_2'(r))'=r[(1+a_2)F_2-a_2F_1]<rF_2
<-\frac{1}{2}re^{u_2(r)}\;\;\text{on}\;\;(R_1, d_n].\ee
This, together with Lemma \ref{lemma2-3}, yields $0\le ru_2'(r)\le R_1u_2'(R_1)\le C$ uniformly on $[R_1, d_n]$.  Then by the mean value theorem, there exists $e_n\in (b_n, d_n)$ such that
$$\frac{1}{d_n-b_n}=\frac{u_2(d_n)-u_2(b_n)}{d_n-b_n}=u_2'(e_n)\le \frac{C}{R_1},$$
which implies $d_n-b_n\to \iy$. Consequently,
\begin{align*}
0&\le d_n u_2'(d_n)=b_n u_2'(b_n)+\int_{b_n}^{d_n}(ru_2')'dr
\le R_1u_2'(R_1)-\frac{1}{2}\int_{b_n}^{d_n}r e^{u_2(r)}dr\\
&\le C-\frac{1}{2}e^{c_0-2}\int_{b_n}^{d_n}rdr\to-\iy\;\;\text{as $n\to\iy$},
\end{align*}
a contradiction.

{\it Step 2.} We claim the existence of constant $C>0$ independent of $\e_n$ such that
\be\label{eq3-3}u_k(r)+2\ln r\le C\;\;\text{uniformly for $r\in [R_1, r_n]$},\;\; k=1,2.\ee

By Step 1, we may assume $u_1<u_2\le u_2(r_n)<\ln\dd$ on $(R_1, r_n)$ for all $n$. Then
$[(1+a_2)ru_1'+a_1ru_2']'=(A-B)rF_1\le-\frac{A-B}{2}re^{u_1}$ and $[a_2ru_1'+(1+a_1)ru_2']'=(A-B)rF_2\le-\frac{A-B}{2}re^{u_2}$ on $[R_1, r_n]$, which imply
{\allowdisplaybreaks
\begin{align*}
0&\le (1+a_2)r_nu_1'(r_n)+a_1r_nu_2'(r_n)\\
&\le (1+a_2)R_1u_1'(R_1)+a_1R_1u_2'(R_1)-\frac{A-B}{2}\int_{R_1}^{r_n}re^{u_1}dr,\\
0&\le a_2r_nu_1'(r_n)+(1+a_1)r_nu_2'(r_n)\\
&\le a_2R_1u_1'(R_1)+(1+a_1)R_1u_2'(R_1)-\frac{A-B}{2}\int_{R_1}^{r_n}re^{u_2}dr.
\end{align*}
}%
By this and Lemma \ref{lemma2-3}, we obtain
\be\label{eq3-4}\int_{R_1}^{r_n}r(e^{u_1}+e^{u_2})dr\le C\;\;\text{for all}\;\,n.\ee
Since $e^{u_1}\le e^{u_2}\le e^{u_2(r_n)}\to 0$ on $[R_1, r_n]$, it follows that
\be\label{eq3-10}\lim_{n\to \iy}\int_{R_1}^{r_n}r(e^{2u_1}+e^{2u_2}+e^{u_1+u_2})dr=0.\ee
Combining \eqref{eq3-10} with Lemma \ref{lemma2-3},  we integrate the Pohozaev identity \eqref{pohozaev} over $(R_1, r)$ for any $r\in (R_1, r_n]$, which yields that
{\allowdisplaybreaks
\begin{align*}
&\frac{A-B}{2}r^2(a_2e^{u_1}+a_1e^{u_2})\le J(ru_1'(r)+2, ru_2'(r)+2)+(A-B)r^2\\
&\quad\times\bigg[a_2e^{u_1}+a_1e^{u_2}
-\frac{a_2(1+a_1)}{2}e^{2u_1}
-\frac{a_1(1+a_2)}{2}e^{2u_2}+a_1a_2e^{u_1+u_2}\bigg]\\
&=J(R_1u_1'(R_1)+2, R_1u_2'(R_1)+2)+o(1)\le C
\end{align*}
}%
holds for all $r\in (R_1, r_n]$. This proves \eqref{eq3-3}.

{\it Step 3.}
Denote $(x_1, y_1):=(2-2\ga,\, 2+\frac{2a_2}{1+a_1}(\ga+N_1)+2N_2)$. Clearly the following system
\begin{equation}\label{eq3-6}
\begin{cases}
J(x, y)=J(x_1, y_1),\\
(1+a_2)x+a_1 y=(1+a_2)x_1+a_1y_1,
\end{cases}
\end{equation}
has at most two distinct solutions, one is just $(x_1, y_1)$, and we denote the other one, if exists, by $(x_2, y_2)$.

Fix any $\theta\in [2.5,\, 3]\setminus\{2-x_1,\, 2-x_2\}$. Since $r_nu_1'(r_n)>0$ and \eqref{eq3-23} gives $R_1u_1'(R_1)<-4$, there exists $t_n\in (R_1, r_n)$ such that $t_nu_1'(t_n)=-\theta$ and $ru_1'(r)<-\theta$ for all $r\in [R_1, t_n)$. We claim the existence of constant $C$ independent of $\e_n$ such that
\be\label{eq3-5}u_2(t_n)+2\ln t_n\ge C\;\;\text{for $n$ sufficiently large}.\ee

Observe that $r^\theta e^{u_1(r)}$ is decreasing for $r\in [R_1, t_n]$, which implies
$$\int_{R_1}^{t_n}re^{u_1}dr\le R_1^\theta e^{u_1(R_1)}\int_{R_1}^{t_n}r^{1-\theta}dr
\le\frac{R_1^2 e^{u_1(R_1)}}{\theta-2}\to 0\;\;\text{as $n\to \iy$}.$$
Since $[(1+a_2)ru_1'+a_1ru_2']'=(A-B)rF_1$ and $|F_1|\le (1+a_1)e^{u_1}$, we can obtain
\begin{align}\label{eq3-18}&(1+a_2)t_nu_1'(t_n)+a_1t_nu_2'(t_n)\\
=&(1+a_2)R_1u_1'(R_1)
+a_1R_1u_2'(R_1)+o(1)\nonumber\end{align}
as $n\to\iy$. This, together with Lemma \ref{lemma2-3}, shows that
$$(1+a_2)(2-\theta)+a_1(2+\lim_{n\to \iy}t_nu_2'(t_n))=(1+a_2)x_1+a_1y_1.$$
Since $2-\theta\not\in\{x_1, x_2\}$, so $J(2-\theta,\, 2+\lim_{n\to \iy}t_nu_2'(t_n))-J(x_1, y_1)\neq 0$. Recall from \eqref{eq3-3} that $t_n^2e^{(u_i+u_j)(t_n)}\le Ct_n^{-2}\to 0$ as $n\to \iy$ for $1\le i, j\le 2$.
Then by \eqref{eq3-10} and Lemma \ref{lemma2-3}, we easily deduce via integrating the Pohozaev identity \eqref{pohozaev} over $(R_1, t_n)$ that
{\allowdisplaybreaks
\begin{align}\label{eq3-19}
&(A-B)t_n^2(a_2e^{u_1(t_n)}+a_1e^{u_2(t_n)})\\
=J&(2+R_1u_1'(R_1), 2+R_1u_2'(R_1))-
J(2+t_nu_1'(t_n), 2+t_nu_2'(t_n))+o(1)\nonumber\\
\to J&(x_1, y_1)- J\left(2-\theta,\, 2+\lim_{n\to \iy}t_nu_2'(t_n)\right)\neq 0\nonumber
\end{align}
}%
as $n\to \iy$. This proves \eqref{eq3-5} since $u_1(t_n)<u_2(t_n)$.

{\it Step 4.} For $k=1, 2$, we consider the scaled functions
$$\hat{u}_{k,n}(r):=u_{k,\e_n}(t_n r)+2\ln t_n,\quad\frac{R_{1,\e_n}}{t_n}\le r\le \frac{r_n}{t_n}.$$
Then $(\hat{u}_{1, n}, \hat{u}_{2, n})$ satisfies
\begin{equation}\label{eq3-7}
\begin{cases}
\hat{u}_{1,n}''+\frac{1}{r}\hat{u}_{1,n}' =(1+a_1)\left((1+a_1)t_n^{-2}e^{2\hat{u}_{1,n}}-e^{\hat{u}_{1,n}}
-a_1t_n^{-2}e^{\hat{u}_{1,n}+\hat{u}_{2,n}}\right)\\
\qquad\qquad\qquad\quad-a_1\left((1+a_2)t_n^{-2}e^{2\hat{u}_{2,n}}-e^{\hat{u}_{2,n}}
-a_2t_n^{-2}e^{\hat{u}_{1,n}+\hat{u}_{2,n}}\right),\\
\hat{u}_{2,n}''+\frac{1}{r}\hat{u}_{2,n}' =(1+a_2)\left((1+a_2)t_n^{-2}e^{2\hat{u}_{2,n}}-e^{\hat{u}_{2,n}}
-a_2t_n^{-2}e^{\hat{u}_{1,n}+\hat{u}_{2,n}}\right)\\
\qquad\qquad\qquad\quad-a_2\left((1+a_1)t_n^{-2}e^{2\hat{u}_{1,n}}-e^{\hat{u}_{1,n}}
-a_1t_n^{-2}e^{\hat{u}_{1,n}+\hat{u}_{2,n}}\right).
\end{cases}
\end{equation}
By \eqref{eq3-3} and \eqref{eq3-5}, we see that
\be\label{eq3-8}|\hat{u}_{2,n}(1)|\le C\;\;\text{for sufficiently large $n$}.\ee

{\it Step 5.} We claim that $\frac{R_{1,\e_n}}{t_n}\to 0$ and $u_{1,\e_n}(t_n)-u_{1,\e_n}(R_{1,\e_n})\to-\iy$ as $n\to \iy$.

Assume by contradiction that up to a subsequence, $t_n/R_1\le C$ for all $n$. Similarly as \eqref{eq3-9}, we can prove that $0\le ru_2'(r)\le R_1u_2'(R_1)\le C$ uniformly on $[R_1, r_n]$. Consequently,
$$u_2(r)-u_2(R_1)\le C\ln\frac{r}{R_1}\le C\;\;\text{uniformly for $r\in [R_1, t_n]$}, $$
which implies from Lemma \ref{lemma2-3} that
\begin{align*}
\int_{R_1}^{t_n}re^{u_1}dr\le \int_{R_1}^{t_n}re^{u_2}dr\le R_1^2e^{u_2(R_1)+C}\left[\left(\frac{t_n}{R_1}\right)^2-1\right]\to 0
\end{align*}
as $n\to\iy$. Recalling the first equation in \eqref{eq5-10} and $|F_k|\le (1+a_k)e^{u_k}$, we obtain
{\allowdisplaybreaks
\begin{align*}
-\theta=t_nu_1'(t_n)&\le R_1u_1'(R_1)+\int_{R_1}^{t_n}r\left[(1+a_1)|F_1|+a_1|F_2|\right]dr\\
&=-2\ga+o(1)\;\;\text{as $n\to \iy$},
\end{align*}
}%
a contradiction with $\theta\le 3$ and $\ga>N_1+2$. This proves $R_1/t_n\to 0$ as $n\to\iy$. Consequently, we deduce from $ru_1'(r)\le -\theta$ on $[R_1, t_n]$ that
$u_1(t_n)-u_1(R_1)\le -\theta\ln\frac{t_n}{R_1}\to -\iy$ as $n\to\iy$.

{\it Step 6.} We claim that $\frac{r_n}{t_n}\to\iy$ as $n\to \iy$.

By \eqref{eq3-3}, \eqref{eq3-10} and the Pohozaev identity \eqref{pohozaev}, it follows that
$$J(ru_1'(r)+2, ru_2'(r)+2)=J(R_1u_1'(R_1)+2, R_1u_2'(R_1)+2)+O(1)\le C$$
uniformly for $r\in [R_1, r_n]$. Hence
\be\label{eq3-11}|ru_1'(r)|+|ru_2'(r)|\le C\;\;\text{uniformly for $r\in [R_1, r_n]$}.\ee
Then by $u_2(r_n)-u_2(R_1)>0$ and Step 5, we have
{\allowdisplaybreaks
\begin{align*}
C\ln\frac{r_n}{t_n}&\ge\int_{t_n}^{r_n}u_1'(r)dr=u_1(r_n)-u_1(t_n)\\
&=u_2(r_n)-u_2(R_1)+u_1(R_1)-u_1(t_n)\to \iy
\end{align*}
}%
as $n\to\iy$. This proves the claim.

{\it Step 7.} We conclude the proof by obtaining a contradiction.

By $u_2(t_n)>u_2(R_1)=u_1(R_1)$, \eqref{eq3-8} and Step 5, we have
\begin{align*}
\hat{u}_{1, n}(1)=u_1(t_n)-u_2(t_n)+\hat{u}_{2, n}(1)\le u_1(t_n)-u_1(R_1)+C\to-\iy
\end{align*}
as $n\to\iy$. Combining this with \eqref{eq3-8} and \eqref{eq3-11}, we conclude that $\hat{u}_{2, n}$ is  uniformly bounded in $C_{loc}((0, \iy))$, while $\hat{u}_{1,n}\to-\iy$ uniformly on any compact subset $K\subset\subset (0, \iy)$ as $n\to\iy$. Up to a subsequence, we may assume that $\hat{u}_{2, n}\to \hat{u}$ in $C_{loc}^2((0,\iy))$, where $\hat{u}$ satisfies
$$
\begin{cases}
\hat{u}''+\frac{1}{r}\hat{u}'=-(1+a_2)e^{\hat{u}}\quad\text{for}\;\,r\in (0, \iy),\\
\int_0^\iy e^{\hat{u}}rdr\le\liminf\limits_{n\to\iy}\int_{R_1}^{r_n}e^{u_2}rdr<+\iy.
\end{cases}
$$
Recalling $u_2'(r)>0$ on $(R_1, r_n)$, we easily conclude that $\hat{u}'(r)\ge 0$ for any $r>0$, namely $\hat{u}$ is increasing on $(0, \iy)$,  which contradicts to $\int_0^\iy re^{\hat{u}}dr<\iy$. This completes the proof.
\end{proof}

\bl\label{lemma2-5} There exists a small $\e_3\in (0, \e_2)$ such that for each $\e\in (0, \e_3)$, there holds $R_{2,\e}<\iy$, $u_{2,\e}'(R_{2,\e})=0$, $u_{1,\e}(R_{2,\e})<u_{2,\e}(R_{2,\e})<\ln\dd$, $|u_{2,\e}(R_{2,\e})+2\ln R_{2,\e}|\le C$ and
\be\label{eq3-12}u_{k,\e}(r)+2\ln r\le C\;\;\text{uniformly for}\;\, r\in [R_{1,\e}, R_{2,\e}],\;\,k=1, 2.\ee
Furthermore, $\frac{R_{2,\e}}{R_{1,\e}}\to\iy$ as $\e\to 0$.\el

\begin{proof}
We divide the proof into four steps.

{\it Step 1.} We claim that $R_{2,\e}<\iy$, $u_{2,\e}'(R_{2,\e})=0$ and $u_{1,\e}(R_{2,\e})<u_{2,\e}(R_{2,\e})<\ln\dd$ for $\e>0$ sufficiently small.

Lemma \ref{lemma2-4} shows $u_1<u_2$ on $(R_1, R_2)$.
By repeating Step 1 of Lemma \ref{lemma2-4}, we can prove
\be\label{eq3-13}\sup_{R_1\le r<R_2}u_2(r)\to -\iy\;\;\text{as}\;\;\e\to 0.\ee
So $u_1<u_2<\ln\dd-1$ on $(R_1, R_2)$ for $\e>0$ small enough. Then \eqref{eq3-9} holds for any $r\in (R_1, R_2)$. Recalling $u_2'>0$ on $[R_1, R_2)$, we have for any $r\in (R_1, R_2)$ that
{\allowdisplaybreaks
\begin{align}\label{eq3-14}
-C\le ru_2'(r)-R_1u_2'(R_1)
\le -\frac{1}{2}\int_{R_1}^r te^{u_2(t)}dt\le -\frac{1}{4}e^{u_2(R_1)}(r^2-R_1^2).
\end{align}
}%
Letting $r\uparrow R_2$, it follows that $R_2<\iy$ and so $u_2'(R_2)=0$.
The proof of Lemma \ref{lemma2-4} also yields $u_1(R_2)<u_2(R_2)$ for $\e>0$ sufficiently small.

{\it Step 2.} We claim that \eqref{eq3-12} holds provided $\e>0$ is sufficiently small.

In fact, since $u_1<u_2<\ln\dd$ on $(R_1, R_2]$,  \eqref{eq3-14} also implies
$\int_{R_1}^{R_2}r(e^{u_1}+e^{u_2})dr\le C$. This fact, together with \eqref{eq3-13}, implies
\be\label{eq3-17}\lim_{\e\to 0}\int_{R_1}^{R_2}r(e^{2u_1}+e^{2u_2}+e^{u_1+u_2})dr=0.\ee
The rest argument is the same as Step 2 of Lemma \ref{lemma2-4}.

{\it Step 3.} We claim that $\frac{R_{2,\e}}{R_{1,\e}}\to \iy$ as $\e\to 0$.

Recalling $ru_2'(r)\le R_1u_2'(R_1)\le C$ for $r\in [R_1, R_2]$, we have $u_2(r)\le u_2(R_1)+C\ln\frac{r}{R_1}$ for all $r\in [R_1, R_2]$. Consequently,
{\allowdisplaybreaks
\begin{align*}
-R_1u_2'(R_1)&=\int_{R_1}^{R_2}(ru_2'(r))'dr=\int_{R_1}^{R_2}r[(1+a_2)F_2-a_2F_1]dr\\
&\ge(1+a_2)\int_{R_1}^{R_2}rF_2dr\ge-(1+a_2)^2\int_{R_1}^{R_2}r e^{u_2(r)}dr\\
&\ge-(1+a_2)^2R_1^2e^{u_2(R_1)}\left[\left(\tfrac{R_2}{R_1}\right)^{C+2}-1\right].
\end{align*}
}%
This proves the claim because Lemma \ref{lemma2-3} gives $R_1^2e^{u_2(R_1)}\to 0$ and $R_1u_2'(R_1)\to C>0$ as $\e\to 0$.

{\it Step 4.} We prove the existence of constant $C$ independent of $\e$ such that $u_{2,\e}(R_{2,\e})+2\ln R_{2,\e}\ge C$ provided $\e>0$ is sufficiently small.

Assume by contradiction that there exist a sequence $\e_n\downarrow 0$ such that $u_{2,\e_n}(R_{2, \e_n})+2\ln R_{2,\e_n}\to -\iy$ as $n\to\iy$. We will omit the subscript $\e_n$ for convenience. Since $u_2$ is increasing on $[R_1, R_2]$, we have
\be\label{eq3-15}r^2e^{u_1(r)}\le r^2e^{u_2(r)}\le R_2^2e^{u_2(R_2)}\to 0\;\;\text{for any $r\in [R_1, R_2]$.}\ee
We consider two cases separately.

{\bf Case 1.} Up to a subsequence, $\sup_{[R_1, R_2]}ru_1'(r)\le -2.5$ for all $n$.

Then $r^{2.5}e^{u_1(r)}$ is decreasing on $[R_1, R_2]$, which implies
\be\label{eq3-16}\int_{R_1}^{R_2}re^{u_1}dr\le 2R_1^2e^{u_1(R_1)}\to 0\;\;\text{as $n\to \iy$}.\ee
Let $(x_1, y_1)$ and $(x_2, y_2)$ be in Step 3 of Lemma \ref{lemma2-4}. Clearly $y_1>2$. Fix any $\theta\in (0, \frac{y_1-2}{2})\setminus\{y_2-2\}$.
Since $u_2'(R_2)=0$ and $u_2'(R_1)=\frac{2a_2}{1+a_1}(\ga+N_1)+2N_2+o(1)=y_1-2+o(1)>\frac{y_1-2}{2}$ for $n$ large, there exists $t_n\in (R_1, R_2)$ such that $t_nu_2'(t_n)=\theta$. By \eqref{eq3-16}  we see that \eqref{eq3-18} holds as $n\to\iy$, which implies
$$(1+a_2)(2+\lim_{n\to\iy}t_nu_1'(t_n))+a_1(2+\theta)=(1+a_2)x_1+a_1y_1.$$
On the other hand, by \eqref{eq3-17} and \eqref{eq3-15}, we can prove via the Pohozaev identity \eqref{pohozaev} that (compare with \eqref{eq3-19})
$$J\left(2+\lim_{n\to \iy}t_nu_1'(t_n),\, 2+\theta \right)-J(x_1, y_1)=0.$$
That is, $(2+\lim_{n\to \iy}t_nu_1'(t_n),\, 2+\theta)$ is also a solution of \eqref{eq3-6}, which yields a contradiction with $2+\theta\not\in\{y_1, y_2\}$. So Case 1 is impossible.

{\bf Case 2. }Up to a subsequence, $\sup_{[R_1, R_2]}ru_1'(r)> -2.5$ for all $n$.

In this case, since $R_1u_1'(R_1)<-4$ by \eqref{eq3-23}, we can repeat the argument of Step 3 in Lemma \ref{lemma2-4} to obtain the existence of $t_n\in (R_1, R_2)$ such that \eqref{eq3-5} holds. Since $u_2$ is increasing on $[R_1, R_2]$, so
$$C\le u_2(t_n)+2\ln t_n\le u_2(R_2)+2\ln R_2\to -\iy$$
as $n\to \iy$, also a contradiction. So Case 2 is also impossible. This completes the proof.
\end{proof}

Lemma \ref{lemma2-5} implies $R_{2,\e}<R_\e^*$ for each $\e\in (0, \e_3)$. Consider the following scaled functions:
\be\label{eq3-20}\bar{u}_k(r)=\bu_{k,\e}(r):=u_{k,\e}(R_{2,\e}r)+2\ln R_{2,\e}\;\,\text{for}\;\, k=1,2,\;\,\e\in (0, \e_3),\ee
where $\frac{R_{1,\e}}{R_{2,\e}}\le r\le 1$.
Then $(\bar{u}_{1}, \bar{u}_{2})$ satisfies
\begin{equation}\label{eq3-21}
\begin{cases}
\bar{u}_{1}''+\frac{1}{r}\bar{u}_{1}' =(1+a_1)\left((1+a_1)R_2^{-2}e^{2\bar{u}_{1}}-e^{\bar{u}_{1}}
-a_1R_2^{-2}e^{\bar{u}_{1}+\bar{u}_{2}}\right)\\
\qquad\qquad\qquad\quad-a_1\left((1+a_2)R_2^{-2}e^{2\bar{u}_{2}}-e^{\bar{u}_{2}}
-a_2R_2^{-2}e^{\bar{u}_{1}+\bar{u}_{2}}\right),\\
\bar{u}_{2}''+\frac{1}{r}\bar{u}_{2}' =(1+a_2)\left((1+a_2)R_2^{-2}e^{2\bar{u}_{2}}-e^{\bar{u}_{2}}
-a_2R_2^{-2}e^{\bar{u}_{1}+\bar{u}_{2}}\right)\\
\qquad\qquad\qquad\quad-a_2\left((1+a_1)R_2^{-2}e^{2\bar{u}_{1}}-e^{\bar{u}_{1}}
-a_1R_2^{-2}e^{\bar{u}_{1}+\bar{u}_{2}}\right).
\end{cases}
\end{equation}
Moreover, $|\bar{u}_2(1)|\le C$ uniformly for $\e$ by Lemma \ref{lemma2-5}. Now we claim that
\be\label{eq3-22}\lim_{\e\to 0}\bar{u}_1(1)=-\iy.\ee

In fact, since $\bar{u}_1\le \bar{u}_2\le \bar{u}_2(1)\le C$ on $[R_1/R_2, 1]$, it follows from \eqref{eq3-21} that $|(r\bar{u}_k')'|\le C$ uniformly for $r\in[R_1/R_2, 1]$ and $k=1, 2$. Consequently,
$$r\bar{u}_1'(r)\le \frac{R_1}{R_2}\bar{u}_1'\left(\frac{R_1}{R_2}\right)
+Cr= R_1u_1'(R_1)+Cr\le -4+Cr,$$
and so $\bar{u}_1(r)\ge -4\ln r+\bar{u}_1(1)-C$ for any $r\in[R_1/R_2, 1]$. Recalling $\int_{R_1}^{R_2}re^{u_1}dr\le C$ uniformly for all $\e\in (0, \e_3)$ by Step 2 of Lemma \ref{lemma2-5}, we have
$$C\ge\int_{R_1}^{R_2}re^{u_1}dr=\int_{\frac{R_1}{R_2}}^1r e^{\bar{u}_1}dr\ge \frac{1}{2}\left(\left(\tfrac{R_2}{R_1}\right)^2-1\right)e^{\bar{u}_1(1)-C}.$$
This proves \eqref{eq3-22} since $R_2/R_1\to\iy$ as $\e\to 0$.

Again by $|(r\bar{u}_k')'|\le C$ on $[R_1/R_2, 1]$ for $k=1, 2$, it follows that $\bar{u}_1\to-\iy$ uniformly on any compact subset $K\subset\subset (0, 1]$ as $\e\to 0$ and $\bar{u}_2$ is uniformly bounded in $C_{loc}((0, 1])$.

\bl\label{lemma2-6}
$\lim\limits_{\e\to 0}\int_{R_{1,\e}}^{R_{2,\e}}r e^{u_{1,\e}}dr=0$. Consequently,
\be\label{eq3-24}\lim_{\e\to 0}R_{2, \e}u_{1,\e}'(R_{2,\e})=-2\frac{A-B}{A}\ga+\frac{2B}{A}N_1+\frac{2a_1}{1+a_2}N_2.\ee
\el

\begin{proof} Assume by contradiction that there exists a sequence $\e_n\downarrow 0$ such that
\be\label{eq3-25}\lim_{n\to\iy}\int_{R_{1,\e_n}}^{R_{2,\e_n}}r e^{u_{1,\e_n}}dr>0.\ee
Again we will omit the subscript $\e_n$ for convenience.
We consider two cases separately.

{\bf Case 1.} Up to a subsequence, $\sup_{[R_1, R_2]}ru_1'(r)\le -2.5$ for all $n$.

Then \eqref{eq3-16} holds, a contradiction with \eqref{eq3-25}. So Case 1 is impossible.

{\bf Case 2. }Up to a subsequence, $\sup_{[R_1, R_2]}ru_1'(r)> -2.5$ for all $n$.

In this case, since $R_1u_1'(R_1)<-4$ by \eqref{eq3-23}, we can repeat the argument of Step 3 in Lemma \ref{lemma2-4}. In particular, there exist a constant $\theta\in (2.5, 3)$ and a sequence $t_n\in (R_1, R_2)$ such that $t_nu_1'(t_n)=-\theta$, $ru_1'(r)<-\theta$ for $r\in [R_1, t_n)$ and $u_2(t_n)+2\ln t_n\ge C$ for $n$ large. Then by the same argument used in \eqref{eq3-16}, we have
\be\label{eq3-26}\int_{R_1}^{t_n}r e^{u_1}dr\le \frac{1}{\theta-2}R_1^2e^{u_1(R_1)}\to 0\;\,\text{as}\;\,n\to\iy.\ee
On the other hand, since $\bar{u}_2\le \bar{u}_2(1)\le C$ on $[R_1/R_2, 1]$, we have
\begin{align*}
C\le t_n^2e^{u_2(t_n)}=\left(\frac{t_n}{R_2}\right)^2e^{\bar{u}_2(\frac{t_n}{R_2})}\le C\left(\frac{t_n}{R_2}\right)^2,
\end{align*}
which implies $t_n/R_2\ge C>0$ for all $n$. Recalling that $\bar{u}_1\to-\iy$ uniformly on $[C, 1]$, we conclude that
$$\int_{t_n}^{R_2}re^{u_1}dr=\int_{\frac{t_n}{R_2}}^1re^{\bar{u}_1}dr\to 0\;\,\text{as}\;\,n\to\iy.$$
Combining this with \eqref{eq3-26}, we obtain a contradiction with \eqref{eq3-25} again.

 Therefore, $\lim_{\e\to 0}\int_{R_{1}}^{R_{2}}r e^{u_{1}}dr=0$. Consequently, by the same argument as \eqref{eq3-18}, we have
$$(1+a_2)R_2u_1'(R_2)+a_1R_2u_2'(R_2)=(1+a_2)R_1u_1'(R_1)
+a_1R_1u_2'(R_1)+o(1).$$
Then \eqref{eq3-24} follows directly from Lemma \ref{lemma2-3} and $u_2'(R_2)=0$.
\end{proof}

For each fixed $\e\in (0, \e_3)$, we define
$$R_3=R_{3,\e}:=\sup\left\{r\in [R_{2, \e}, R_\e^*)\,|\, u_1<u_2\;\text{on}\; [R_{2,\e}, r)\right\}.$$
Then $R_{2}<R_{3}\le R^*$. If there exists $t\in (R_{2}, R_{3})$
such that $u_{2}'(t)=0$ and $u_{2}'(r)<0$ for $r\in (R_{2}, t)$, then Lemma \ref{lemma2-5} yields $u_1<u_2<\ln\dd$ on $[R_{2}, t]$. Consequently, $F_2<F_1<0$ on $[R_{2}, t]$ and so
$$0=t u_{2}'(t)-R_{2}u_{2}'(R_{2})
=\int_{R_{2}}^{t}r[(1+a_2)F_2-a_2F_1]dr<0,$$
a contradiction. Therefore,
\be\label{eq3-27}u_{2}'(r)<0\;\,\text{for any}\;\,r\in(R_2, R_3).\ee

Consider the scaled functions $\bu_k$ defined in \eqref{eq3-20} for $r\in (\frac{R_1}{R_2}, \frac{R_3}{R_2})$. By \eqref{eq3-27} and the definition of $R_3$, we have $\bu_1(r)<\bu_2(r)\le \bu_2(1)\le C$ for all $r\in (\frac{R_1}{R_2}, \frac{R_3}{R_2})$. This, together with \eqref{eq3-21}, gives $|(r\bu_k')'(r)|\le Cr$ for all $r\in (\frac{R_1}{R_2}, \frac{R_3}{R_2})$. Consequently, $\bar{u}_1\to-\iy$ uniformly on any compact subset $K\subset\subset (\frac{R_1}{R_2}, \frac{R_3}{R_2})$ as $\e\to 0$ and $\bar{u}_2$ is uniformly bounded in $C_{loc}((\frac{R_1}{R_2}, \frac{R_3}{R_2}))$.
Since $\bu_2(1)-\bu_1(1)\to\iy$ as $\e\to 0$, we conclude from the definition of $R_3$ that
\be\label{eq3-28}\lim_{\e\to 0}\frac{R_{3}}{R_{2}}=\iy.\ee
Then, for any constant $b>1$, there holds
\be\label{eq3-28-1}\lim_{\e\to 0}\int_{R_1}^{bR_2}re^{u_1}dr=\lim_{\e\to 0}\int_{R_1}^{R_2}re^{u_1}dr+\lim_{\e\to 0}\int_{1}^{b}re^{\bu_1}dr=0.\ee

\bl\label{lemma2-7}$\bu_{2,\e}\to \om_2$ in $C_{loc}^2((0,\iy))$ as $\e\to 0$, where
\be\label{eq3-29}\om_2(r)=\ln\frac{2D^2(D^2-4)r^{D-2}}{(1+a_2)(D+2+(D-2)r^D)^2}
\;\;\text{for}\;\,r\in (0,\iy).\ee
Here $D:=\frac{2a_2}{1+a_1}(\ga+N_1)+2N_2+2$. Consequently,
\be\label{eq3-30}\int_{0}^\iy re^{\om_2}dr=\frac{2}{1+a_2}D.\ee
\el

\begin{proof}
Recall \eqref{eq3-17} and Lemma \ref{lemma2-3}. By integrating the Pohozaev identity over $(R_1, R_2)$, we obtain
{\allowdisplaybreaks
\begin{align*}
&J(R_2u_1'(R_2)+2, R_2u_2'(R_2)+2)+
(A-B)R_2^2\bigg[a_2e^{u_1(R_2)}+a_1e^{u_2(R_2)}\\
&\qquad-\frac{a_2(1+a_1)}{2}e^{2u_1(R_2)}
-\frac{a_1(1+a_2)}{2}e^{2u_2(R_2)}+a_1a_2e^{u_1(R_2)+u_2(R_2)}\bigg]\\
=&J(R_1u_1'(R_1)+2, R_1u_2'(R_1)+2)+o(1),\;\;\text{as}\;\,\e\to 0.
\end{align*}
}%
Recall that $R_2^2e^{u_1(R_2)}=e^{\bu_1(1)}\to 0$ and $R_2^2e^{u_i(R_2)+u_j(R_2)}=R_2^{-2}e^{\bu_i(1)+\bu_j(1)}\to 0$ for $1\le i, j\le 2$ as $\e\to 0$.
Combining these with Lemma \ref{lemma2-3} and \eqref{eq3-24}, we have
{\allowdisplaybreaks
\begin{align}\label{eq3-48}
&a_1(A-B)R_2^2e^{u_2(R_2)}\\
=&J(R_1u_1'(R_1)+2, R_1u_2'(R_1)+2)\nonumber\\
&-J(R_2u_1'(R_2)+2, R_2u_2'(R_2)+2)+o(1)\nonumber\\
=&J\Big(2-2\ga,\; \underbrace{\frac{2a_2}{1+a_1}(\ga+N_1)+2N_2+2}_{=:D}\Big)\nonumber\\
&-J\bigg(\underbrace{2-2\frac{A-B}{A}\ga+\frac{2B}{A}N_1+\frac{2a_1}{1+a_2}N_2}
_{=\frac{a_1}{1+a_2}D-(2\ga-2)-\frac{2a_1}{1+a_2}},\; 2\bigg)+o(1)\nonumber\\
=&\frac{a_1(A-B)}{2(1+a_2)}(D^2-4)+o(1),\;\;\text{as}\;\,\e\to 0.\nonumber
\end{align}
}%
Hence $e^{\bu_2(1)}=R_2^2e^{u_2(R_2)}\to \frac{1}{2(1+a_2)}(D^2-4)$ as $\e\to 0$.

Recalling that $\bar{u}_2$ is uniformly bounded in $C_{loc}((\frac{R_1}{R_2}, \frac{R_3}{R_2}))$, up to a subsequence, we may assume that $\bu_2\to\om_2$ in $C_{loc}^2((0,\iy))$. Clearly, $\om_2$ satisfies
\be\label{eq3-31}
\begin{cases}
\begin{split}
&\om_2''+\frac{1}{r}\om_2'=-(1+a_2)e^{\om_2},\\
&\om_2(r)\le \om_2(1)=\ln\frac{D^2-4}{2(1+a_2)},\\
\end{split}\quad\text{for}\;\,r\in (0, \iy).
\end{cases}
\ee
Consequently, a simple contrary argument shows $\int_{0}^\iy re^{\om_2}dr<\iy$.
Recalling $0\le r\bu_2'(r)= R_2ru_2'(R_2r)\le R_1u_2'(R_1)=D-2+o(1)$ for all $r\in [\frac{R_1}{R_2}, 1]$, we obtain (note $\om_2'(1)=0$)
$$2\ga_2:=\lim_{r\to 0}r\om_2'(r)\in (0, D-2].$$ In conclusion,
$\om_2$ is a radial solution of the Liouville equation with singular sources
$$\Delta v+(1+a_2)e^v=4\pi\ga_2\dd_0\;\;\text{in}\;\,\R^2, \quad\int_0^\iy re^vdr<\iy.$$
By a well-known classification result due to Prajapat and Tarantello \cite{PT}, there holds
$$\om_2(r)+\ln(1+a_2)=\ln\frac{8\la(1+\ga_2)^2 r^{2\ga_2}}{(1+\la r^{2\ga_2+2})^2}$$
for some constant $\la>0$. By $\om_2'(1)=0$ and $\om_2(1)=\ln\frac{D^2-4}{2(1+a_2)}$, a direct computation gives $\ga_2=\frac{D-2}{2}$ and $\la=\frac{D-2}{D+2}$. Consequently, we see that \eqref{eq3-29}-\eqref{eq3-30} hold. The above argument actually shows that $\bu_2\to \om_2$ in $C_{loc}^2((0,\iy))$ as $\e\to 0$ (i.e., not only along a subsequence). This completes the proof.
\end{proof}

\bl\label{lemma2-8}There exists a small $\e_4\in (0, \e_3)$ such that for each $\e\in (0, \e_4)$, there holds $R_{3,\e}<R_\e^*$. Consequently, $u_{1,\e}(R_{3,\e})=u_{2,\e}(R_{3,\e})$.\el

\begin{proof}
Assume by contradiction that there exists a sequence $\e_n\downarrow 0$ such that $R_{3,\e_n}=R_{\e_n}^*$. Since $u_{1,\e_n}<u_{2,\e_n}\le u_{2,\e_n}(R_{2,\e_n})<\ln\dd$ on $(R_{1,\e_n}, R_{3, \e_n})$, we see from the definition of $R_{\e_n}^*$ that $R_{3,\e_n}=\iy$, namely $(u_{1,\e_n}, u_{2,\e_n})$ is an entire solution and $u_{1,\e_n}<u_{2,\e_n}<\ln\dd$ on $(R_{1, \e_n}, \iy)$. By Theorem A, we see that $(u_{1,\e_n}, u_{2,\e_n})$ is a non-topological solution and there exist constants $\bb_{k,\e_n}>1$ such that $ru_{k,\e_n}'(r)\to -2\bb_{k,\e_n}$ as $r\to \iy$ for $k=1, 2$. Clearly, $\bb_{2,\e_n}\le \bb_{1,\e_n}$ for all $n$. Again we will omit the subscript $\e_n$ for convenience.

By Lemma \ref{lemma2-7}, we can fix a large constant $b>1$ such that $b^2e^{\om_2(b)}<\dd$ and $\int_{b}^{\iy}re^{\om_2}dr<\dd/2$. By the dominated convergence theorem,
$$\int_{R_1}^{bR_2}re^{u_2}dr=\int_{\frac{R_1}{R_2}}^{b}re^{\bu_2}dr\to \int_{0}^{b}re^{\om_2}dr\;\,\text{as}\;\,n\to \iy.$$
Recall \eqref{eq3-28-1}, \eqref{eq3-30} and $u_1<u_2\le u_2(R_2)\to -\iy$ on $(R_1, \iy)$. Then for $n$ sufficiently large, we have
{\allowdisplaybreaks
\begin{align*}
&|\bu_2(b)-\om_2(b)|<\dd,\quad \left|\int_{R_1}^{bR_2}re^{u_2}dr-\frac{2}{1+a_2}D\right|<\dd,\\
&(A-B)^2\int_{R_1}^{bR_2}r(e^{u_1}+e^{2u_1}+e^{2u_2}+e^{u_1+u_2})dr<\dd.
\end{align*}
}%
Recalling \eqref{eq5-10} and \eqref{eq4-0}, we have
{\allowdisplaybreaks
\begin{align*}
bR_2u_1'(bR_2)\ge& R_1u_1'(R_1)+a_1\int_{R_1}^{bR_2}r e^{u_2}dr\\
&-\int_{R_1}^{bR_2}r\left[(1+a_1)e^{u_1}+a_1(1+a_1)e^{u_1+u_2}
+a_1(1+a_2)e^{2u_2}\right]dr\\
\ge & -2\ga-\dd+a_1\left(\frac{2}{1+a_2}D-\dd\right)-\dd\\
=&-2\frac{A-2B}{A}\ga+\frac{4B}{A}N_1+\frac{4a_1}{1+a_2}(N_2+1)-(2+a_1)\dd,
\end{align*}
}%
and
{\allowdisplaybreaks
\begin{align}\label{eq3-35}
bR_2u_2'(bR_2)\le& R_1u_2'(R_1)-(1+a_2)\int_{R_1}^{bR_2}r e^{u_2}dr\nonumber\\
&+\int_{R_1}^{bR_2}r\left[(1+a_2)^2e^{2u_2}+a_2e^{u_1}
+a_1a_2e^{u_1+u_2}\right]dr\nonumber\\
\le & \frac{2a_2}{1+a_1}(\ga+N_1)+2N_2+\dd-(1+a_2)\left(\frac{2}{1+a_2}D
-\dd\right)+\dd\nonumber\\
=&-\frac{2a_2}{1+a_1}(\ga+N_1)-2N_2-4+(3+a_2)\dd\nonumber\\
<& -4-\frac{4a_2}{1+a_1}-2(2+a_1)\dd,
\end{align}
}%
where we have used $\frac{2a_2}{1+a_1}(\ga-2)>7(1+a_1+a_2)\dd$ (by \eqref{eq3-1}) to obtain the last inequality.
Recalling $u_1<u_2\le u_2(R_2)<\ln\dd$ on $[R_2, \iy)$, we have $F_2<F_1<0$ and so $(ru_2')'(r)=r[(1+a_2)F_2-a_2F_1]<0$ on $[R_2, \iy)$. Consequently, $ru_2'(r)\le bR_2u_2'(bR_2)$ for any $r\ge bR_2$, which implies
\be\label{eq3-32}-2\bb_2\le-4-\frac{4a_2}{1+a_1}-2(2+a_1)\dd.\ee
On the other hand, by $ru_2'(r)\le bR_2u_2'(bR_2)<-4$ for $r\ge bR_2$, we also have
{\allowdisplaybreaks
\begin{align}\label{eq3-36}
\int_{bR_2}^\iy r e^{u_1}dr&\le\int_{bR_2}^\iy r e^{u_2}dr \le (bR_2)^4e^{u_2(bR_2)}\int_{bR_2}^\iy r^{-3}dr\nonumber\\
&=\frac{1}{2}b^2e^{\bu_2(b)}
\le\frac{1}{2}b^2e^{\om_2(b)+\dd}\le\frac{1}{2}\dd e^{\dd}<\dd.
\end{align}
}%
Consequently,
{\allowdisplaybreaks
\begin{align*}
-2\bb_1&=bR_2u_1'(bR_2)+\int_{bR_2}^\iy (ru_1'(r))'dr\\
&\ge bR_2u_1'(bR_2)\\
&\quad-\int_{bR_2}^\iy
r\left[(1+a_1)e^{u_1}+a_1(1+a_1)e^{u_1+u_2}+a_1(1+a_2)e^{2u_2}\right]dr\\
&\ge bR_2u_1'(bR_2)-(2+a_1)\dd\\
&\ge -2\frac{A-2B}{A}\ga+\frac{4B}{A}N_1+\frac{4a_1}{1+a_2}(N_2+1)-2(2+a_1)\dd.
\end{align*}
}%
This, together with \eqref{eq3-32} and $\bb_2\le \bb_1$, gives
$$\frac{A-2B}{A}(\ga-1)\ge \frac{2B}{A}(N_1+1)+\frac{2a_1}{1+a_2}(N_2+1)+1+\frac{2a_2}{1+a_1},$$
However, $\al_1\ge 1$ and \eqref{eq2-5} give
$$\frac{A-2B}{A}(\ga-1)\le \frac{2B}{A}(N_1+1)+\frac{2a_1}{1+a_2}(N_2+1),$$
which yields a contradiction. This completes the proof.
\end{proof}

\begin{lemma}\label{lemma2-9} There hold
$\lim\limits_{\e\to 0}\int_{R_{1,\e}}^{R_{3,\e}}r e^{u_{1,\e}}dr=0$, $\lim\limits_{\e\to 0}\int_{R_{1,\e}}^{R_{3,\e}}r e^{u_{2,\e}}dr=\frac{2}{1+a_2}D$ and $\lim\limits_{\e\to 0}R_{3,\e}^2 e^{u_{k,\e}(R_{3,\e})}=0$ for $k=1, 2$. Consequently,
{\allowdisplaybreaks
\begin{align}\label{eq3-33}
&\lim_{\e\to 0}R_{3,\e}u_{1,\e}'(R_{3,\e})=-2\frac{A-2B}{A}\ga+\frac{4B}{A}N_1+
\frac{4a_1}{1+a_2}(N_2+1)\ge -2,\\
\label{eq3-34}&\lim_{\e\to 0}R_{3,\e}u_{2,\e}'(R_{3,\e})=-\frac{2a_2}{1+a_1}(\ga+N_1)-
2N_2-4<-4.
\end{align}
}%
Furthermore, $\lim\limits_{\e\to 0}R_{3,\e}u_{1,\e}'(R_{3,\e})=-2$ if and only if $A-2B>0$ and $\al_1=1$.
\end{lemma}

\begin{proof}
Given any $\mu\in (0, \dd)$, there exists a large constant $b_\mu>1$ such that $b_\mu^2e^{\om_2(b_\mu)}<\mu$ and $\int_{b_\mu}^{\iy}re^{\om_2}dr<\mu/2$. Then by a similar argument used in Lemma \ref{lemma2-8}, we have for $\e>0$ sufficiently small that
{\allowdisplaybreaks
\begin{align*}
&|\bu_2(b_\mu)-\om_2(b_\mu)|<\mu,\quad \left|\int_{R_1}^{b_\mu R_2}re^{u_2}dr-\frac{2}{1+a_2}D\right|<\mu,\\
&(A-B)^2\int_{R_1}^{b_\mu R_2}r(e^{u_1}+e^{2u_1}+e^{2u_2}+e^{u_1+u_2})dr<\mu,\\
&\left|R_1u_2'(R_1)-\frac{2a_2}{1+a_1}(\ga+N_1)-2N_2\right|<\mu.
\end{align*}
}%
Consequently, by repeating the argument of \eqref{eq3-35}-\eqref{eq3-36}, we can prove
$$\int_{b_\mu R_2}^{R_3} r e^{u_1}dr\le\int_{b_\mu R_2}^{R_3} r e^{u_2}dr \le\frac{1}{2}b_\mu^2e^{\om_2(b_\mu)+\mu}<\mu,$$
and
\begin{align*}
R_3^2e^{u_1(R_3)}=R_3^2e^{u_2(R_3)}<(b_\mu R_2)^2e^{u_2(b_\mu R_2)}=
b_\mu^2e^{\bu_2(b_\mu)}\le b_\mu^2e^{\om_2(b_\mu)+\mu}\le 2\mu,
\end{align*}
namely $R_3^2e^{u_k(R_3)}<2\mu$, $\int_{R_1}^{R_3}re^{u_1}dr<2\mu$ and $|\int_{R_1}^{ R_3}re^{u_2}dr-\frac{2}{1+a_2}D|<2\mu$ for $\e>0$ sufficiently small. This proves
\be\label{eq3-37}\lim\limits_{\e\to 0}R_3^2e^{u_k(R_3)}=0,\;\lim\limits_{\e\to 0}\int_{R_{1}}^{R_{3}}r e^{u_{1}}dr=0,\;\lim\limits_{\e\to 0}\int_{R_{1}}^{R_{3}}r e^{u_{2}}dr=\frac{2}{1+a_2}D.\ee
Recalling $u_1<u_2\le u_2(R_2)\to -\iy$ on $(R_1, R_3)$ as $\e\to 0$, we have
$$\lim_{\e\to 0}\int_{R_1}^{R_3}r\left(e^{2u_1}+e^{2u_2}+e^{u_1+u_2}\right)dr=0.$$
Consequently, we integrate \eqref{eq5-10} over $(R_1, R_3)$ to derive
{\allowdisplaybreaks
\begin{align*}
&\lim_{\e\to 0}R_3u_1'(R_3)=\lim_{\e\to 0}\left[R_1u_1'(R_1)+a_1\int_{R_1}^{R_3}re^{u_2}dr\right]
=-2\ga+\frac{2a_1}{1+a_2}D\\
&\quad=-2\frac{A-2B}{A}\ga+\frac{4B}{A}N_1+
\frac{4a_1}{1+a_2}(N_2+1)\ge-2,\;\,\text{(by \eqref{eq2-5-1})}\\
&\lim_{\e\to 0}R_3u_2'(R_3)=\lim_{\e\to 0}\left[R_1u_2'(R_1)-(1+a_2)\int_{R_1}^{R_3}re^{u_2}dr\right]\\
&\quad=\frac{2a_2}{1+a_1}(\ga+N_1)+2N_2-2D
=-\frac{2a_2}{1+a_1}(\ga+N_1)-
2N_2-4.
\end{align*}
}%
This completes the proof.
\end{proof}

Fix any constant $\vartheta\in [0, \dd)$ such that
\be\label{eq6-1}\vartheta=0\;\;\text{if}\;\,\al_1>1\quad
\text{and}\quad\vartheta>0\;\;\text{if}\;\,\al_1=1.\ee
Then by Lemma \ref{lemma2-9}, there exists $\e_5\in (0, \e_4)$ such that $R_{3,\e}^2e^{u_{1,\e}(R_{3,\e})}<\dd$ and $R_{3,\e}u_1'(R_{3,\e})>-2-\vartheta>R_{3,\e}u_2'(R_{3,\e})+1$ for any $\e\in (0, \e_5)$. For each $\e\in (0, \e_5)$, we define
{\allowdisplaybreaks
\begin{align*}
&R_4=R_{4,\e}:=\sup\left\{r\in [R_{3,\e}, R_\e^*)\,|\, ru_{1,\e}'(r)>-2-\vartheta\right\},\\
&R_5=R_{5,\e}:=\sup\left\{r\in (R_{3,\e}, R_\e^*)\,|\, u_{1,\e}>u_{2,\e}\;\,\text{on}\;\,(R_{3,\e}, r)\right\}.
\end{align*}
}%
Clearly, $R_{4,\e}, R_{5,\e}>R_{3,\e}$ for all $\e\in (0, \e_5)$.

\bl\label{lemma2-10} There exists a small $\e_6\in (0, \e_5)$ such that for each $\e\in (0, \e_6)$, $u_{2,\e}<u_{1,\e}<\ln\dd$ on $(R_{3,\e}, R_{5,\e})$, $ru_{1,\e}'(r)$ is strictly decreasing on $(R_{3,\e}, R_{5,\e})$ and $R_{4,\e}\le R_{5,\e}$. In particular, if $\al_1>1$, then $R_{4,\e}<R_{5, \e}$, namely $R_{4,\e}u_{1,\e}'(R_{4,\e})=-2$, $ru_{1,\e}'(r)>-2$ on $[R_{3,\e}, R_{4,\e})$ and $ru_{1,\e}'(r)<-2$ on $(R_{4,\e}, R_{5,\e})$. \el

\begin{proof} We divide the proof into four steps.

{\it Step 1.} We claim that
\be\label{eq3-38}\sup_{[R_{3,\e}, R_{5,\e})}u_{1,\e}\to-\iy\;\;\text{as}\;\,\e\to 0.\ee

Suppose by contradiction that there exist a sequence $\e_n\downarrow 0$ and a constant $c_0<\ln\dd$ such that $\sup_{[R_{3,\e_n}, R_{5,\e_n})}u_{1,\e_n}\ge c_0$ for all $n$. We will omit the subscript $\e_n$ for convenience. The following proof is similar to Step 1 of Lemma \ref{lemma2-4}. For $n$ large, there exist $b_n, d_n\in (R_3, R_5)$ such that $b_n<d_n$, $u_1(d_n)=c_0-1$, $u_1(b_n)=c_0-2$, $u_1\le c_0-1$ on $[R_3, d_n]$ and $u_1\ge c_0-2$ on $[b_n, d_n]$. Clearly, $u_1'(d_n)\ge 0$, $u_2<u_1<\ln\dd$ and so $F_1<F_2<0$ on $(R_3, d_n]$, which implies
\be\label{eq3-9-1}(ru_1'(r))'=r[(1+a_1)F_1-a_1F_2]<rF_1
<-\frac{1}{2}re^{u_1(r)}\;\;\text{on}\;\;(R_3, d_n].\ee
Then $0\le ru_1'(r)\le R_3u_1'(R_3)\le C$ for any $r\in[R_3, d_n]$, which yields $d_n-b_n\to \iy$. Consequently,
\begin{align*}
0&\le d_n u_1'(d_n)=b_n u_1'(b_n)+\int_{b_n}^{d_n}(ru_1')'dr
\le R_3u_1'(R_3)-\frac{1}{2}\int_{b_n}^{d_n}re^{u_1}dr\\
&\le C-\frac{1}{2}e^{c_0-2}\int_{b_n}^{d_n}rdr\to-\iy\;\;\text{as $n\to\iy$},
\end{align*}
a contradiction.

{\it Step 2.} By Step 1, for $\e>0$ sufficiently small, we have $u_2<u_1<\ln\dd$
on $(R_3, R_5)$, which implies that \eqref{eq3-9-1} holds on $(R_3, R_5)$ and so $ru_1'(r)$ is strictly decreasing on $[R_3, R_5)$.

{\it Step 3.} We prove that $R_{4,\e}\le R_{5,\e}$ for $\e>0$ sufficiently small.

Assume by contradiction that there exist a sequence $\e_n\downarrow 0$ such that
$R_{5,\e_n}<R_{4,\e_n}$. Again we omit the subscript $\e_n$ for convenience. Consequently, $R_5<\iy$ and $ru_1'(r)>-2-\vartheta$ for $r\in [R_3, R_5]$. Since $[a_2ru_1'+(1+a_1)ru_2']'=(A-B)rF_2<0$ on $[R_3, R_5]$, we have
$$ru_2'(r)\le \frac{a_2}{1+a_1}\left[R_3u_1'(R_3)+2+\vartheta\right]+R_3u_2'(R_3)=:l_n$$
uniformly for $r\in [R_3, R_5]$. Recalling \eqref{eq2-7} and \eqref{eq3-33}-\eqref{eq3-34}, it is easy to see that
\begin{align*}l:=&\lim_{n\to \iy}l_n\\
=&\frac{4a_2(B-A)}{(1+a_1)A}\tilde{\ga}
-\frac{2a_2(A-2B)}{(1+a_1)A}\tilde{N}_1-2\frac{A-2B}{A}\tilde{N}_2-2
+\frac{a_2\vartheta}{1+a_1}.\end{align*}

We claim $l<-2-\dd$. Recall $3A-4B>0$, $\tilde{\ga}>1$, $\vartheta<\dd$ and \eqref{eq3-1}. Clearly $l<-2-\frac{a_2}{1+a_1}\frac{A-B}{A}\tilde{\ga}<-2-\dd$ if $A\ge 2B$. Let us consider the remaining case $\frac{4}{3}B<A<2B$. By $\al_2>1$ and \eqref{eq2-6}, we easily obtain
\begin{align}\label{eq3-42}
\frac{a_2}{1+a_1}\tilde{\ga}
&>\frac{a_2}{1+a_1}\frac{4B-A}{3A-4B}\tilde{N}_1+\frac{4B-A}{3A-4B}\tilde{N}_2\nonumber\\
&>\frac{a_2}{1+a_1}\frac{2B-A}{A-B}\tilde{N}_1+\frac{2B-A}{A-B}\tilde{N}_2,
\end{align}
where we have used $\frac{4B-A}{3A-4B}>\frac{2B-A}{A-B}$. So we also get $l<-2-\frac{a_2}{1+a_1}\frac{A-B}{A}\tilde{\ga}<-2-\dd$.

Hence, for $n$ sufficiently large, $ru_2'(r)\le l_n<-2-\vartheta< ru_1'(r)$ for all $r\in [R_3, R_5]$, namely $u_2-u_1$ is strictly decreasing on $[R_3, R_5]$, which contradicts to $u_2(R_5)-u_1(R_5)=0$.

{\it Step 4.} Let $\al_1>1$, then $\vartheta=0$. We prove that $R_4< R_5$ for $\e>0$ sufficiently small.

Assume by contradiction that $R_4=\iy$ for some $\e>0$ sufficiently small. Then $R_5=\iy$ and so  $ru_1'(r)>-2$ for all $r\ge R_3$, which implies $\int_{R_3}^\iy re^{u_1}dr=\iy$. On the other hand, Step 2 shows that \eqref{eq3-9-1} holds on $(R_3, \iy)$, so
$$\iy=\frac{1}{2}\int_{R_3}^\iy re^{u_1}dr\le\limsup_{r\to\iy}[-tu_1'(t)]\Big|_{R_3}^r
\le R_3u_1'(R_3)+2\le C,$$
a contradiction. So $R_4<\iy$ for $\e>0$ sufficiently small. Then by repeating the argument of Step 3, we finally conclude that $R_4< R_5$ for $\e>0$ sufficiently small.
\end{proof}

\bl\label{lemma2-11}
For each $\e\in (0,\e_6)$, $u_{k,\e}(r)+2\ln r\le C$ uniformly for $r\in [R_{3,\e}, R_{5,\e})$ and $k=1, 2$. Furthermore, $\lim\limits_{\e\to 0}\int_{R_{3,\e}}^{R_{5,\e}}r e^{u_{2,\e}}dr=0$ and
{\allowdisplaybreaks
\begin{align}\label{eq3-39}&\lim_{\e\to 0}[a_2ru_{1,\e}'(r)+(1+a_1)ru_{2,\e}'(r)]\nonumber\\
=&-4a_2\frac{A-B}{A}\ga-2a_2\frac{A-2B}{A}N_1-2\frac{A-2B}{1+a_2}N_2-4\frac{A-B}{1+a_2}
\end{align}
}%
uniformly for all $r\in [R_{3,\e}, R_{5,\e})$.\el

\begin{proof} We separate the proof into three steps.

{\it Step 1.} For each $\e\in (0,\e_6)$, we claim that
\be\label{eq3-40}u_k(r)+2\ln r\le C\;\,\text{uniformly for}\;\,r\in [R_3, R_5).\ee

Lemma \ref{lemma2-10} shows that $u_2<u_1<\ln\dd$ and $(ru_1')'\le-\frac{1}{2}re^{u_1}$ on $[R_3, R_5)$. Hence
$$-2-\vartheta<ru_1'(r)\le R_3u_2'(R_3)-\frac{1}{2}\int_{R_3}^{r}re^{u_1}dr,
\quad\forall\,r\in (R_3, R_4),$$
which implies
$$\int_{R_3}^{R_4}r\left(e^{u_1}+e^{u_2}\right)dr\le 2\int_{R_3}^{R_4}re^{u_1}dr\le C.$$
This, together with \eqref{eq3-38}, gives
$$\lim_{\e\to 0}\int_{R_3}^{R_4}r\left[e^{2u_1}+e^{2u_2}+e^{u_1+u_2}\right]dr=0.$$
Then by repeating the argument of Step 2 in Lemma \ref{lemma2-4}, we have $u_k(r)+2\ln r\le C$ for all $r\in [R_3, R_4)$ and $k=1, 2$. If $R_4=R_5$, we are done. If $R_4<R_5$, by $ru_1'(r)<-2-\vartheta$ for $r\in (R_4, R_5)$, we conclude that
$$u_2(r)+2\ln r<u_1(r)+2\ln r\le u_1(R_4)+2\ln R_4\le C$$
for all $r\in [R_4, R_5)$. This proves \eqref{eq3-40}.

{\it Step 2.} Recalling \eqref{eq2-6} and $\tilde{\al}_2=\al_2-1>0$, we claim that for $\e>0$ sufficiently small,
\be\label{eq3-41}ru_2'(r)\le -2-\frac{3}{2}\tilde{\al}_2\;\,\text{uniformly for}\;\,r\in [R_3, R_5).\ee

By \eqref{eq3-40} we have
\be\label{eq3-41-1}\lim_{\e\to 0}\int_{R_3}^{R_5}r\left[e^{2u_1}+e^{2u_2}+e^{u_1+u_2}\right]dr=0.\ee
Consequently, integrating the Pohozaev identity \eqref{pohozaev} over $[R_3, r]$ gives
{\allowdisplaybreaks
\begin{align}\label{eq3-43}
J(ru_1'(r)+2, ru_2'(r)+2)&\le J(ru_1'(r)+2, ru_2'(r)+2)+(A-B)r^2\bigg[a_2e^{u_1}\nonumber\\
\quad+a_1e^{u_2}
-\frac{a_2(1+a_1)}{2}&e^{2u_1}
-\frac{a_1(1+a_2)}{2}e^{2u_2}+a_1a_2e^{u_1+u_2}\bigg]\nonumber\\
&=J(R_3u_1'(R_3)+2, R_3u_2'(R_3)+2)+o(1)
\end{align}
}%
uniformly for all $r\in [R_3, R_5)$ as $\e\to 0$.
On the other hand, since $[a_2ru_1'+(1+a_1)ru_2']'=(A-B)rF_2<0$ on $[R_3, R_5)$,
we have for any $r\in [R_3, R_5)$ that
{\allowdisplaybreaks
\begin{align}\label{eq3-44}
&a_2(ru_1'(r)+2)+(1+a_1)(ru_2'(r)+2)\nonumber\\
\le& a_2(R_3u_1'(R_3)+2)+(1+a_1)(R_3u_2'(R_3)+2)=:\eta_\e.
\end{align}
}%
Recalling \eqref{eq2-7} and \eqref{eq3-33}-\eqref{eq3-34}, we have
$$\eta:=\lim_{\e\to 0}\eta_\e
=-4a_2\frac{A-B}{A}\tilde{\ga}-2a_2\frac{A-2B}{A}\tilde{N}_1
-2\frac{A-2B}{1+a_2}\tilde{N}_2.$$
Using \eqref{eq3-42} if $A<2B$, we easily see $\eta<0$. Hence $\eta_\e<0$ for $\e>0$ sufficiently small. Note from \eqref{eq1-9} that
$$J(x, y)=\frac{a_2(A-B)}{2(1+a_1)}x^2+\frac{a_1}{2(1+a_1)}[a_2x+(1+a_1)y]^2.$$
This, together with \eqref{eq3-43}-\eqref{eq3-44} and $\eta_\e<0$, easily yields
$$[ru_1'(r)+2]^2\le [R_3u_1'(R_3)+2]^2+o(1)$$
and so $ru_1'(r)\ge -R_3 u_1'(R_3)-4+o(1)$ uniformly for any $r\in [R_3, R_5)$ as $\e\to 0$. Substituting this inequality into \eqref{eq3-44} and recalling \eqref{eq3-33}-\eqref{eq3-34}, we finally obtain
{\allowdisplaybreaks
\begin{align*}
ru_2'(r)&\le \frac{2a_2}{1+a_1}R_3u_1'(R_3)+R_3u_2'(R_3)+\frac{4a_2}{1+a_1}+o(1)\\
&=-\frac{2a_2(3A-4B)}{(1+a_1)A}\tilde{\ga}-\frac{2a_2(A-4B)}{(1+a_1)A}\tilde{N}_1
-2\frac{A-4B}{A}\tilde{N}_2-2+o(1)\\
&=-2\tilde{\al}_2-2+o(1) \;\;(\text{by \eqref{eq2-6}})
\end{align*}
}%
uniformly for any $r\in [R_3, R_5)$ as $\e\to 0$. This proves \eqref{eq3-41}.

{\it Step 3.} We prove $\lim\limits_{\e\to 0}\int_{R_{3,\e}}^{R_{5,\e}}r e^{u_{2,\e}}dr=0$ and \eqref{eq3-39}.

By \eqref{eq3-41}, $r^{2+\tilde{\al}_2}e^{u_2(r)}$ is strictly decreasing for $r\in [R_3, R_5)$, so Lemma \ref{lemma2-9} gives
$$\int_{R_3}^{R_5}re^{u_2}dr\le \frac{1}{\tilde{\al}_2}R_3^2e^{u_2(R_3)}\to 0\;\;\text{as}\;\,\e\to 0.$$
Then by integrating $[a_2ru_1'+(1+a_1)ru_2']'=(A-B)rF_2$ over $[R_3, r]$ for any $r\in [R_3, R_5)$ and recalling $|F_2|\le (1+a_2)e^{u_2}$, we easily obtain \eqref{eq3-39}. This completes the proof.
\end{proof}

Now we consider the cases $\al_1=1$ and $\al_1>1$ separately.

\subsection{The critical case $\al_1=1$}

In this subsection, we consider the critical case $\al_1=1$. Consequently,
we see from \eqref{eq2-4-1}, \eqref{eq2-5} and \eqref{eq2-7} that
\be\label{eq6-2}A-2B>0\;\;\text{and}\;\;
\tilde{\ga}=\frac{2B}{A-2B}\tilde{N}_1+\frac{2a_1(1+a_1)}{A-2B}\tilde{N}_2.\ee
The following lemma provides an evidence that this critical case is different from the generic case $\al_1>1$.

\bl\label{lemma3-1}There exists a small $\e_7\in (0, \e_6)$ such that for each $\e\in (0, \e_7)$, $R_{4,\e}=R_{5, \e}=R_\e^*=+\iy$.\el

\begin{proof} Assume by contradiction that $R_{4,\e_n}<+\iy$ for a sequence $\e_n\downarrow 0$. We will omit the subscript $\e_n$ for convenience. Then $R_4u_1'(R_4)=-2-\vartheta$.
This, together with \eqref{eq3-39} and \eqref{eq6-2}, gives
\begin{align}&\label{eq6-3}\lim_{n\to \iy}R_4u_2'(R_4)+2\nonumber\\
=&\frac{4a_2(B-A)}{(1+a_1)A}\tilde{\ga}
-\frac{2a_2(A-2B)}{(1+a_1)A}\tilde{N}_1
-2\frac{A-2B}{A}\tilde{N}_2+\frac{a_2\vartheta}{1+a_1}\nonumber\\
=&-\frac{1+a_2}{a_1}\tilde{\ga}+\frac{a_2}{1+a_1}\vartheta.\end{align}
Similarly, by Lemma \ref{lemma2-9} we have $\lim_{n\to \iy}R_3u_1'(R_3)+2=0$ and
\begin{align*}
\lim_{n\to\iy}R_3u_2'(R_3)+2=-\frac{2a_2}{1+a_1}(\tilde{\ga}+\tilde{N}_1)-2\tilde{N}_2
=-\frac{1+a_2}{a_1}\tilde{\ga}.
\end{align*}
Since we have assumed $R_4<\iy$, it follows from
\eqref{eq3-40} and \eqref{eq3-41} that $R_4^2e^{u_i(R_4)+u_j(R_4)}\le CR_4^{-2}\to 0$ for $1\le i,j\le 2$ and
$R_4^2e^{u_2(R_4)}<R_3^2e^{u_2(R_3)}\to 0$ as $n\to\iy$.
Combining these with \eqref{eq3-41-1}, we can repeat the proof of Lemma \ref{lemma2-7} to obtain (similar to \eqref{eq3-48})
{\allowdisplaybreaks
\begin{align*}
&a_2(A-B)R_4^2e^{u_1(R_4)}\\
=&J(R_3u_1'(R_3)+2, R_3u_2'(R_3)+2)-J(R_4u_1'(R_4)+2, R_4u_2'(R_4)+2)+o(1)\\
=&J\left(0,\; -\frac{1+a_2}{a_1}\tilde{\ga}
\right)-J\left(-\vartheta, \; -\frac{1+a_2}{a_1}\tilde{\ga}+\frac{a_2}{1+a_1}\vartheta\right)+o(1)\\
=&-\frac{a_2(A-B)}{2(1+a_1)}\vartheta^2+o(1)\;\;\text{as}\;\,n\to \iy,
\end{align*}
}%
which yields a contradiction with $\vartheta>0$.\end{proof}

Now we can finish the proof of Theorem \ref{th2-1} for $\al_1=1$.

\begin{proof}[Completion of the proof of Theorem \ref{th2-1} for $\al_1=1$.]
Let $\e\in (0, \e_7)$, then $R_{4, \e}=R_{5,\e}=R_\e^*=+\iy$. Since
Lemma \ref{lemma2-10} shows that $u_{2,\e}(r)<u_{1,\e}(r)<\ln\dd$ for any $r\in (R_{3,\e}, +\iy)$, we conclude that $(u_{1,\e}, u_{2,\e})$ is an entire solution. By Theorem A, there exists $(\al_{1,\e}, \al_{2,\e})\in\Om$ such that
$$u_{k,\e}(r)=-2\al_{k,\e}\ln r+O(1)\;\,\text{as}\;\,r\to\iy, \;\;k=1,2.$$
Consequently, $ru_{k,\e}'(r)\to -2\al_{k,\e}$ as $r\to\iy$. Then Lemma \ref{lemma2-10} and the definition of $R_{4,\e}$ yield $-2-\vartheta\le -2\al_{1,\e}<R_{3,\e}u_{1,\e}'(R_{3,\e})$, namely
$$2+\vartheta\ge\limsup_{\e\to 0}2\al_{1,\e}\ge\liminf_{\e\to 0}2\al_{1,\e}\ge-\lim_{\e\to 0}R_{3,\e}u_{1,\e}'(R_{3,\e})=2.$$
Since $\vartheta\in (0, \dd)$ can be taken apriori arbitrary small, we conclude that
$$\lim_{\e\to 0}\al_{1,\e}=\al_1=1.$$
This, together with \eqref{eq3-39}, easily implies{\allowdisplaybreaks
\begin{align*}&a_2+(1+a_1)\lim_{\e\to 0}\al_{2,\e}\\
=&2a_2\frac{A-B}{A}\ga+a_2\frac{A-2B}{A}N_1
+\frac{A-2B}{1+a_2}N_2+2\frac{A-B}{1+a_2},\end{align*}
}so
{\allowdisplaybreaks
\begin{align*}
\lim_{\e\to 0}\al_{2,\e}&=\frac{2a_2}{1+a_1}\frac{A-B}{A}\tilde{\ga}
+\frac{a_2}{1+a_1}\frac{A-2B}{A}\tilde{N}_1
+\frac{A-2B}{A}\tilde{N}_2+1\\
&=\frac{1+a_2}{2a_1}\tilde{\ga}+1\quad\text{(by \eqref{eq6-3})}\\
&=\al_2. \quad\text{(by \eqref{eq2-4})}
\end{align*}
}%

By Step 3 in the proof of Lemma \ref{lemma2-11}, we have $\int_{R_{3,\e}}^\iy r F_2dr\to 0$ as $\e\to 0$.
Consequently, by integrating $(ru_{1,\e}')'=(1+a_1)r F_1-a_1 rF_{2}$ over $(R_{3,\e}, +\iy)$, we deduce from $-2\al_{1,\e}-R_{3,\e}u_{1,\e}'(R_{3,\e})\to 0$ that $\int_{R_{3,\e}}^\iy r F_1dr\to 0$ as $\e\to 0$. Then by \eqref{eq5-10-2} we conclude that
\be\label{eq6-4}\lim_{\e\to 0}\int_{R_{3,\e}}^\iy r e^{u_{1,\e}}dr=0.\ee

Observe from \eqref{eq3-27} and \eqref{eq3-41} that $u_{2,\e}'(r)<0$ for all $r>R_{2,\e}$. Besides, Lemma \ref{lemma2-3} shows that $\sup_{[0, R_{1,\e}]}u_{2,\e}\to-\iy$. Combining these with \eqref{eq3-13}, we conclude $\sup_{\R^2}u_{2,\e}\to -\iy$ as $\e\to 0$.
This completes the proof.
\end{proof}

\br\label{remark2-1}
In Theorem C where bubbling solutions of type I are constructed,
we assumed $\al_1>1$, which plays an essential role in the proof of Theorem C (see \cite{CL}). In particular,
the conclusion $(\al_{1,\e}, \al_{2,\e})\to (\al_1, \al_2)$
is a corollary of \eqref{eq6-4} in the proof of Theorem C.
However, for the critical case $\al_1=1$ studied here, the idea used in Theorem C can not be applied and we have to argue the other way around: the conclusion \eqref{eq6-4} is a consequence of $(\al_{1,\e}, \al_{2,\e})\to (\al_1, \al_2)$.\er

\subsection{The generic case $\al_1>1$}

In this subsection, we consider the generic case $\al_1>1$. Then $\vartheta=0$ and $R_{4,\e}<R_{5,\e}$. The following lemma also provides an evidence that this case is different from the critical case $\al_1=1$.

\bl\label{lemma2-12} Recalling $\al_1>1$ in \eqref{eq2-5}, there holds (compare to \eqref{eq6-4})
\begin{align}\label{eq3-45}\lim_{\e\to 0}\int_{R_{3,\e}}^{R_{5,\e}}r e^{u_1}dr
&=\frac{4}{1+a_1}\left(\frac{2B-A}{A}\tilde{\ga}
+\frac{2B}{A}\tilde{N}_1+\frac{2a_1}{1+a_2}\tilde{N}_2\right)\nonumber\\
&=\frac{4(\al_1-1)}{1+a_1}.\end{align}\el

\begin{proof}
We separate the proof into five steps.

{\it Step 1.} We claim that
\be\label{eq3-46}\lim_{\e\to 0}R_{4,\e}^2e^{u_{1,\e}(R_{4,\e})}=\frac{E^2}{2(1+a_1)},\ee
where $E>0$ is defined by
\begin{align}\label{eq3-47}E:=&\lim_{\e\to 0}R_{3,\e}u_{1,\e}'(R_{3,\e})+2\\
=&2\frac{2B-A}{A}\tilde{\ga}
+\frac{4B}{A}\tilde{N}_1+\frac{4a_1}{1+a_2}\tilde{N}_2=2(\al_1-1). \nonumber\end{align}

Recalling \eqref{eq3-39} and $R_4u_1'(R_4)=-2$, there holds
$$F:=\lim_{\e\to 0}R_4u_2'(R_4)+2=-\frac{4a_2(A-B)}{(1+a_1)A}\tilde{\ga}
-\frac{2a_2(A-2B)}{(1+a_1)A}\tilde{N}_1-2\frac{A-2B}{A}\tilde{N}_2.$$
Again, \eqref{eq3-40} and \eqref{eq3-41} imply $R_4^2e^{u_i(R_4)+u_j(R_4)}\le CR_4^{-2}\to 0$ for $1\le i,j\le 2$ and
\be\label{eq3-51}R_4^2e^{u_2(R_4)}<R_3^2e^{u_2(R_3)}\to 0\;\;\text{as}\;\;\e\to 0.\ee Combining these with \eqref{eq3-41-1}, we can repeat the proof of Lemma \ref{lemma2-7} to obtain (similar to \eqref{eq3-48})
{\allowdisplaybreaks
\begin{align*}
&a_2(A-B)R_4^2e^{u_1(R_4)}\\
=&J(R_3u_1'(R_3)+2, R_3u_2'(R_3)+2)-J(R_4u_1'(R_4)+2, R_4u_2'(R_4)+2)+o(1)\\
=&J\bigg(E,\; \underbrace{-\frac{2a_2}{1+a_1}(\tilde{\ga}+\tilde{N}_1)-2\tilde{N}_2}_{=:G}
\bigg)-J(0, F)+o(1)\\
=&J\left(-2\tilde{\ga}-\frac{2a_1}{1+a_2}G,\; G\right)-J\left(0,\; \frac{A-2B}{A}G-\frac{2a_2}{1+a_1}\tilde{\ga}\right)+o(1)\\
=&J\left(2\tilde{\ga},\, G\right)-J\left(0, \; \frac{A-2B}{A}G-\frac{2a_2}{1+a_1}\tilde{\ga}\right)+o(1)\;\,(\text{by \eqref{eq2-8}})\\
=&\frac{a_2(A-B)}{2(1+a_1)}\left(\frac{2a_1}{1+a_2}G+2\tilde{\ga}\right)^2+o(1)\\
=&\frac{a_2(A-B)}{2(1+a_1)}E^2+o(1)\;\;\text{as}\;\,\e\to 0.
\end{align*}
}%
Hence \eqref{eq3-46} holds.

{\it Step 2.} For $\e>0$ sufficiently small, we consider the scaled functions
$$\tu_{k}(r)=\tu_{k,\e}(r):=u_{k,\e}(R_{4,\e}r)+2\ln R_{4,\e}\;\,\text{for}\;\, k=1,2,$$
where $\frac{R_{3,\e}}{R_{4,\e}}\le r<\frac{R_{5,\e}}{R_{4,\e}}$.
Then $(\tu_1, \tu_2)$ satisfies
\begin{equation}\label{eq3-49}
\begin{cases}
\tu_1''+\frac{1}{r}\tu_1' =(1+a_1)\left((1+a_1)R_4^{-2}e^{2\tu_1}-e^{\tu_1}
-a_1R_4^{-2}e^{\tu_1+\tu_2}\right)\\
\qquad\qquad\qquad-a_1\left((1+a_2)R_4^{-2}e^{2\tu_2}-e^{\tu_2}
-a_2R_4^{-2}e^{\tu_1+\tu_2}\right),\\
\tu_2''+\frac{1}{r}\tu_2' =(1+a_2)\left((1+a_2)R_4^{-2}e^{2\tu_2}-e^{\tu_2}
-a_2R_4^{-2}e^{\tu_1+\tu_2}\right)\\
\qquad\qquad\qquad-a_2\left((1+a_1)R_4^{-2}e^{2\tu_1}-e^{\tu_1}
-a_1R_4^{-2}e^{\tu_1+\tu_2}\right).
\end{cases}
\end{equation}
By Lemma \ref{lemma2-10}, it is easy to see that $u_1(r)+2\ln r\le u_1(R_4)+2\ln R_4$ for any $r\in [R_3, R_5)$. So for any $\frac{R_{3}}{R_{4}}< r<\frac{R_{5}}{R_{4}}$ we have
\be\label{eq3-50}\tu_2(r)+2\ln r<\tu_1(r)+2\ln r\le \tu_1(1)=\ln\frac{E^2}{2(1+a_1)}+o(1).\ee
Moreover, \eqref{eq3-51} gives $\tu_2(1)\to-\iy$ as $\e\to 0$.

{\it Step 3.} We claim that
$$\frac{R_3}{R_4}\to 0\;\;\text{and}\;\;\frac{R_5}{R_4}\to \iy\;\;\text{as}\;\;\e\to 0.$$

Recall from Lemma \ref{lemma2-10} that
$-2\le r\tu_1'(r)=R_4 ru_1'(R_4r)\le R_3u_1'(R_3)\le C$ for $r\in [\frac{R_3}{R_4}, 1]$.
By the mean value theorem, we have
\begin{align*}
\tu_1(1)-\ln \left(R_3^2e^{u_1(R_3)}\right)-2\ln \frac{R_4}{R_3}=\tu_1(1)-\tu_1\left(\frac{R_3}{R_4}\right)\le C\left(\frac{R_4}{R_3}-1\right).
\end{align*}
Recalling $R_3^2e^{u_1(R_3)}\to 0$ and \eqref{eq3-50}, we conclude that $\frac{R_3}{R_4}\to 0$ as $\e\to 0$.

By \eqref{eq3-50} and $R_4^{-2}e^{\tu_k(r)}=e^{u_k(R_4 r)}<1$ for any $r\in [\frac{R_3}{R_4}, \frac{R_5}{R_4})$, it is easy to deduce from \eqref{eq3-49} that
$|(r\tu_k')'(r)|\le C/r$ for all $r\in [\frac{R_3}{R_4}, \frac{R_5}{R_4})$. Consequently,
$\tu_1$ is uniformly bounded in $C_{loc}((\frac{R_3}{R_4}, \frac{R_5}{R_4}))$ and $\tu_2\to -\iy$ uniformly on any compact subset $K\subset\subset (\frac{R_3}{R_4}, \frac{R_5}{R_4})$. This, together with the definition of $R_5$, yields $\frac{R_5}{R_4}\to\iy$ as $\e\to 0$.

{\it Step 4.} We claim that $\tu_1\to \om_1$ in $C_{loc}^2((0, \iy))$ as $\e\to 0$, where
\be\label{eq3-52-1}\om_1(r)=\ln\frac{2E^2r^{E-2}}{(1+a_1)(1+r^E)^2}\;\;\text{for}\;\;r\in
(0, \iy), \ee
and $E$ is seen in \eqref{eq3-47}. Consequently,
\be\label{eq3-52-2}\int_{0}^\iy re^{\om_1}dr=\frac{2}{1+a_1}E.\ee

By Step 3, up to a subsequence, we may assume that $\tu_1\to \om_1$ in $C_{loc}^2((0, \iy))$, where $\om_1$ satisfies
\be\label{eq3-53}
\begin{cases}
\begin{split}
&\om_1''+\frac{1}{r}\om_1'=-(1+a_1)e^{\om_1},\\
&\om_1(r)+2\ln r\le \om_1(1)=\ln\frac{E^2}{2(1+a_1)},\\
\end{split}\quad\text{for}\;\,r\in (0, \iy).
\end{cases}
\ee
Since $\om_1'(1)=-2$ and $r\om_1'(r)$ is strictly decreasing on $(0, \iy)$,
it is easy to prove that $\int_{0}^\iy re^{\om_1}dr<\iy$.
Recalling $-2\le r\tu_1'(r)=R_4ru_1'(R_4r)\le R_3u_1'(R_3)=E-2+o(1)$ for all $r\in [R_3/R_4, 1]$, we easily obtain
$$2\ga_1:=\lim_{r\to 0}r\om_1'(r)\in (-2, E-2].$$ In conclusion,
$\om_1$ is a radial solution of the Liouville equation with singular sources
$$\Delta v+(1+a_1)e^v=4\pi\ga_1\dd_0\;\;\text{in}\;\,\R^2, \quad\int_0^\iy re^vdr<\iy.$$
Since $\ga_1>-1$,  again by the classification result due to Prajapat and Tarantello \cite{PT}, there holds
$$\om_1(r)+\ln(1+a_1)=\ln\frac{8\la(1+\ga_1)^2 r^{2\ga_1}}{(1+\la r^{2\ga_1+2})^2}$$
for some constant $\la>0$. By $\om_1'(1)=-2$ and $\om_1(1)=\ln\frac{E^2}{2(1+a_1)}$, a direct computation gives $\ga_1=\frac{E-2}{2}$ and $\la=1$. This proves \eqref{eq3-52-1} and \eqref{eq3-52-2}.
Clearly, the above argument also shows that $\tu_1\to \om_1$ in $C_{loc}^2((0,\iy))$ as $\e\to 0$ (i.e., not only along a subsequence).

{\it Step 5.} We prove \eqref{eq3-45}.

Given any $\mu\in (0, \dd)$. By \eqref{eq3-52-1}-\eqref{eq3-52-2}, there exist small constant $b_{\mu}\in (0, 1)$ and large constant $d_{\mu}>1$ such that
{\allowdisplaybreaks
\begin{align*}
&b_\mu^2 e^{\om_1(b_\mu)}+ d_\mu^2 e^{\om_1(d_\mu)}+
\left|\int_{b_\mu}^{d_\mu}r e^{\om_1}dr-\frac{2E}{1+a_1}\right|<\frac{\mu}{2},\\
&b_\mu \om_1'(b_\mu)+2\ge\frac{2}{3}E,\quad d_\mu \om_1'(d_\mu)+2\le-\frac{2}{3}E.
\end{align*}
}%
Consequently, since $\tu_1\to \om_1$ in $C^2([b_\mu, d_\mu])$, there exists sufficiently small $\e_\mu>0$ such that for each $\e\in (0, \e_\mu)$, we have
{\allowdisplaybreaks
\begin{align*}
&b_\mu^2 e^{\tu_1(b_\mu)}+ d_\mu^2 e^{\tu_1(d_\mu)}+
\left|\int_{b_\mu}^{d_\mu}r e^{\tu_1}dr-\frac{2E}{1+a_1}\right|<\mu,\\
&b_\mu \tu_1'(b_\mu)+2\ge\frac{1}{2}E,\quad d_\mu \tu_1'(d_\mu)+2\le-\frac{1}{2}E.
\end{align*}
}%
Recalling that $r\tu_1'(r)=R_4r u_1'(R_4r)$ is strictly decreasing on $(\frac{R_3}{R_4}, \frac{R_5}{R_4})$, we obtain $r\tu'(r)+2\le-\frac{1}{2}E$ for all $r\in [d_\mu, \frac{R_5}{R_4})$, which implies that $r^{2+\frac{1}{2}E}e^{\tu_1(r)}$ decreases on $[d_\mu, \frac{R_5}{R_4})$ and so
\[
(R_4r)^2e^{u_1(R_4r)}=r^2e^{\tu_1(r)}\le d_\mu^2 e^{\tu_1(d_\mu)}<\mu\;\;\text{for any}\;\;r\in\left[d_\mu, \tfrac{R_5}{R_4}\right).
\]
Furthermore,
$$\int_{d_\mu R_4}^{R_5}re^{u_1}dr=\int_{d_\mu}^{\frac{R_5}{R_4}}re^{\tu_1}dr\le
\frac{2}{E}d_\mu^2 e^{\tu_1(d_\mu)}<\frac{2}{E}\mu.$$
Similarly, by $r\tu'(r)+2\ge\frac{1}{2}E$ for all $r\in [\frac{R_3}{R_4}, b_\mu]$, we can prove
$$\int_{R_3}^{b_\mu R_4}re^{u_1}dr=\int_{\frac{R_3}{R_4}}^{b_\mu}re^{\tu_1}dr\le
\frac{2}{E}b_\mu^2 e^{\tu_1(b_\mu)}<\frac{2}{E}\mu.$$
Since $\int_{b_\mu R_4}^{d_\mu R_4}r e^{u_1}dr=\int_{b_\mu}^{d_\mu}r e^{\tu_1}dr$, \eqref{eq3-45} follows immediately. This completes the proof.
\end{proof}

\bl\label{lemma2-13}
Given any $t_\e\in (R_{4,\e}, R_{5, \e}]$ such that $t_\e<\iy$ and $\frac{t_\e}{R_{4,\e}}\to \iy$ as $\e\to 0$. Then $\lim\limits_{\e\to 0}t_\e^2 e^{u_{k,\e}(t_\e)}=0$ for $k=1, 2$ and
$$\lim_{\e\to 0}t_\e u_{1,\e}'(t_\e)=-2\al_1,\quad \lim_{\e\to 0}t_\e u_{2,\e}'(t_\e)=-2\al_2,$$
where $\al_1, \al_2$ are seen in \eqref{eq2-5}-\eqref{eq2-6}.
\el

\begin{proof} Since $\frac{t}{R_4}\to \iy$, by repeating the argument of Step 5 in Lemma \ref{lemma2-12}, it is easy to prove that $t^2e^{u_2(t)}\le t^2 e^{u_1(t)}\to 0$ and
$$\int_{R_3}^{t}re^{u_1}dr\to \frac{2E}{1+a_1}\quad\text{as}\;\;\e\to 0.$$
Recalling $\int_{R_3}^{R_5}re^{u_2}dr\to 0$ (see Lemma \ref{lemma2-11}), \eqref{eq3-41-1}, \eqref{eq3-33}-\eqref{eq3-34} and \eqref{eq3-47}, we conclude
{\allowdisplaybreaks
\begin{align*}
\lim_{\e\to 0}tu_1'(t)&=\lim_{\e\to 0}\left[R_3u_1'(R_3)-(1+a_1)\int_{R_3}^t re^{u_1}dr\right]=E-2-2E\\
&=-2\al_1,\\
\lim_{\e\to 0}tu_2'(t)&=\lim_{\e\to 0}\left[R_3u_2'(R_3)+a_2\int_{R_3}^t re^{u_1}dr\right]\\
&=-\frac{2a_2}{1+a_1}(\tilde{\ga}+\tilde{N}_1)-
2\tilde{N}_2-2+\frac{2a_2}{1+a_1}E\\
&=-\frac{2a_2}{1+a_1}\frac{3A-4B}{A}\tilde{\ga}-\frac{2a_2}{1+a_1}\frac{A-4B}{A}\tilde{N}_1
-2\frac{A-4B}{A}\tilde{N}_2-2\\
&=-2\al_2.
\end{align*}
}%
This completes the proof.
\end{proof}

\bl\label{lemma2-14}There exists a small $\e_7\in (0, \e_6)$ such that for each $\e\in (0, \e_7)$, there holds $ru_{2,\e}'(r)<-2-\tilde{\al}_2$ for all $r\in [R_{3,\e}, R_{\e}^*)$. Consequently,
\be\label{eq3-55-1}\lim_{\e\to 0}\int_{R_{3,\e}}^{R_\e^*}re^{u_{2,\e}}dr\le\frac{1}{\tilde{\al}_2}\lim_{\e\to 0}R_{3,\e}^2e^{u_{2,\e}(R_{3,\e})}= 0.\ee
\el

\begin{proof} Recall that $\tilde{\al}_k=\al_k-1>0$ for $k=1, 2$ and \eqref{eq3-41} gives that $ru_{2,\e}'(r)\le -2-\frac{3}{2}\tilde{\al}_2$ for $r\in [R_{3,\e}, R_{5,\e})$. Assume by contradiction that there exist a sequence $\e_n\downarrow 0$ and $t_n\in [R_{5, \e_n}, R_{\e_n}^*)$ such that $t_n u_{2,\e_n}'(t_n)=-2-\tilde{\al}_2$ and $ru_{2,\e}'(r)<-2-\tilde{\al}_2$ for $r\in [R_{3,\e_n}, t_n)$. We will omit the subscript $\e_n$ for convenience.
Clearly $R_5<t_n<R^*$. Since $\frac{R_5}{R_4}\to\iy$, Lemma \ref{lemma2-13} yields $\lim_{n\to \iy}R_5^2e^{u_k(R_5)}= 0$ for $k=1, 2$ and
\be\label{eq3-54}\lim_{n\to \iy}R_5 u_1'(R_5)=-2-2\tilde{\al}_1,\quad \lim_{n\to \iy}R_5 u_{2}'(R_5)=-2-2\tilde{\al}_2.\ee
Since $ru_2'(r)<-2-\tilde{\al}_2$ for $r\in [R_3, t_n)$, we have
\be\label{eq3-55}\int_{R_3}^{t_n}re^{u_2}dr\le\frac{1}{\tilde{\al}_2}R_3^2e^{u_2(R_3)}\to 0\;\;\text{as}\;\,n\to\iy.\ee
Recall from \eqref{eq3-38} that $\sup_{[R_3, R_5]}u_1\to-\iy$ as $n\to\iy$. Since $u_2\le u_1$ on $[R_3, R_5]$, we have $\sup_{[R_3, t_n]}u_2\to-\iy$ as $n\to\iy$.

{\it Step 1.} We claim that $u_1'<0$ on $[R_5, t_n]$ for $n$ sufficiently large. Consequently, $\sup_{[R_3, t_n]}u_1\to-\iy$ as $n\to\iy$.

Suppose that, up to a subsequence, there exists $r_n\in [R_5, t_n]$ such that $u_1'(r_n)=0$ and $u_1'(r)<0$ for $r\in [R_5, r_n)$. Then for $n$ sufficiently large, $u_1<\ln\dd$ and so $F_1<0$, $-F_2<2e^{u_2}$ on $[R_5, r_n]$, which imply that $(ru_1')'=r[(1+a_1)F_1-a_1F_2]\le 2a_1re^{u_2}$ on $[R_5, r_n]$. Consequently,
\begin{align*}
0-R_5u_1'(R_5)=\int_{R_5}^{r_n}(ru_1')'dr\le2a_1\int_{R_5}^{r_n}re^{u_2}dr
\le2a_1\int_{R_3}^{t_n}re^{u_2}dr\to 0
\end{align*}
as $n\to\iy$, which yields a contradiction with \eqref{eq3-54}.

{\it Step 2.} We claim that $ru_1'(r)<-2-\tilde{\al}_1$ on $[R_5, t_n]$ for $n$ large enough. Consequently,
\be\label{eq3-56}\int_{R_5}^{t_n}re^{u_1}dr\le\frac{1}{\tilde{\al}_1}R_5^2e^{u_1(R_5)}\to 0\;\;\text{as}\;\,n\to\iy.\ee

By Step 1 and \eqref{eq3-54}-\eqref{eq3-55}, we may take $n$ large enough such that $(ru_1')'\le 2a_1re^{u_2}$ on $[R_5, t_n]$,  $R_5u_1'(R_5)<-2-\frac{3}{2}\tilde{\al}_1$ and $2a_1\int_{R_5}^{t_n}re^{u_2}dr<\frac{1}{2}\tilde{\al}_1$.  Then for any $r\in [R_5, t_n]$, we have $$ru_1'(r)\le R_5u_1'(R_5)+2a_1\int_{R_5}^{r}te^{u_2}dt<-2-\tilde{\al}_1.$$

{\it Step 3.} We complete the proof.

Similarly as Steps 1-2, we may take $n$ large enough such that $(ru_2')'=r[(1+a_2)F_2-a_2F_1]\le 2a_2re^{u_1}$ on $[R_5, t_n]$, so we conclude from \eqref{eq3-54} and \eqref{eq3-56} that
\begin{align*}
-2-\tilde{\al}_2=t_n u_2'(t_n)\le R_5u_2'(R_5)+2a_1\int_{R_5}^{t_n}re^{u_1}dr\to -2-2\tilde{\al}_2
\end{align*}
as $n\to\iy$, a contradiction with $\tilde{\al}_2>0$. This completes the proof.
\end{proof}

We are now in a position to complete the proof of Theorem \ref{th2-1} for $\al_1>1$.

\begin{proof}[Completion of the proof of Theorem \ref{th2-1} for $\al_1>1$]
For each $\e\in (0, \e_7)$, we take a number $t_\e\in (R_{4,\e}, R_{5, \e})$ such that
$\frac{t_\e}{R_{4,\e}}\to\iy$ as $\e\to 0$. Then Lemma \ref{lemma2-10} gives $u_{2,\e}(t_\e)<u_{1,\e}(t_\e)<\ln\dd$. Moreover, by Lemma \ref{lemma2-13} we have
$$\lim_{\e\to 0}t_\e^2 e^{u_{k,\e}(t_\e)}=0,\; \lim_{\e\to 0}t_\e u_{1,\e}'(t_\e)=-2\al_1,\; \lim_{\e\to 0}t_\e u_{2,\e}'(t_\e)=-2\al_2.$$
Combining these with \eqref{eq3-55-1}, we can repeat the argument of Steps 1-2 in Lemma \ref{lemma2-14} to conclude the existence of $\e_8\in (0, \e_7)$ such that $ru_{1,\e}'(r)<-2-\tilde{\al}_1$ on $[t_\e, R_{\e}^*)$ for any $\e\in (0, \e_8)$. Consequently,
\be\label{eq3-57}\int_{t_\e}^{R_\e^*}re^{u_{1,\e}}dr\le\frac{1}{\tilde{\al}_1}t_{\e}^2
e^{u_1(t_{\e})}\to 0\;\,\text{as}\;\,\e\to 0. \ee

Now we consider $\e\in (0, \e_8)$.
Recall from Lemma \ref{lemma2-14} that $ru_{2,\e}'(r)<-2-\tilde{\al}_2$ on $[R_{3,\e}, R_{\e}^*)$. It turns out that both $u_{1,\e}$ and $u_{2,\e}$ decrease on $[t_\e, R_\e^*)$, which implies $u_{k,\e}(r)<u_{k,\e}(t_\e)<\ln\dd$ for any $r\in (t_\e, R_\e^*)$ and $k=1, 2$. By the definition of $R_\e^*$, we conclude that $R_\e^*=\iy$, namely $(u_{1,\e}, u_{2,\e})$ is an entire solution for $\e\in (0, \e_8)$. By Theorem A, there exists $(\al_{1,\e}, \al_{2,\e})\in\Om$ such that
$$u_{k,\e}(r)=-2\al_{k,\e}\ln r+O(1)\;\,\text{as}\;\,r\to\iy, \;\;k=1,2.$$
Recall that $|(ru_{k,\e}')'|=r|(1+a_k)F_k-a_k F_{3-k}|\le C r(e^{u_{1,\e}}+e^{u_{2,\e}})$ for $k=1, 2$. By \eqref{eq3-55-1}, \eqref{eq3-57} and $R_\e^*=\iy$, we obtain
$$\left|-2\al_{k,\e}-t_\e u_{k,\e}'(t_\e)\right|\le C\int_{t_\e}^\iy r(e^{u_{1,\e}}+e^{u_{2,\e}})dr\to 0\;\,\text{as}\;\,\e\to 0$$
for $k=1, 2$. Therefore, $\lim\limits_{\e\to 0}(\al_{1,\e}, \al_{2,\e})=(\al_1, \al_2)$.

Observe from \eqref{eq3-27} and Lemma \ref{lemma2-14} that $u_{2,\e}'(r)<0$ for all $r>R_{2,\e}$. Besides, Lemma \ref{lemma2-3} shows that $\sup_{[0, R_{1,\e}]}u_{2,\e}\to-\iy$. Combining these with \eqref{eq3-13}, we conclude $\sup_{\R^2}u_{2,\e}\to -\iy$ as $\e\to 0$.

This completes the proof.
\end{proof}

\medskip
\noindent{\bf Acknowledgements} The authors thank the referee very much for careful reading and valuable comments.

\end{document}